\documentclass[11pt]{article}
\usepackage[width=15cm,height=21cm]{geometry}
\usepackage[colorlinks,allcolors=blue]{hyperref}

\makeatletter
\newcommand{\doitext}{doi:}
\newcommand{\doiurl}{https://doi.org/}

\newcommand*{\doi}{%
  \begingroup 
  \lccode`\~=`\#\relax 
  \lowercase{\def~{\#}}%
  \lccode`\~=`\_\relax
  \lowercase{\def~{\_}}%
  \lccode`\~=`\<\relax 
  \lowercase{\def~{\textless}}%
  \lccode`\~=`\>\relax 
  \lowercase{\def~{\textgreater}}%
  \lccode`\~=0\relax 
  \catcode`\#=\active 
  \catcode`\_=\active 
  \catcode`\<=\active 
  \catcode`\>=\active 
  \@doi
}

\def\@doi#1{%
  \let\#\relax
  \let\_\relax
  \let\textless\relax 
  \let\textgreater\relax 
  \edef\x{\toks0={{#1}}}%
  \x
  \edef\#{\@percentchar23}%
  \edef\_{_}%
  \edef\textless{\@percentchar3C}
  \edef\textgreater{\@percentchar3E}
  \edef\x{\toks2={\noexpand\href{\doiurl#1}}}%
  \x
  \edef\x{\endgroup\doitext\the\toks2 \the\toks0}%
  \x
}
\makeatother

\usepackage{cite}
\newcommand{\myurl}[1]{\href{#1}{#1}}
\newcommand{\journaltitle}[1]{\emph{#1}}
\newcommand{\booktitle}[1]{\emph{#1}}
\newcommand{\myvolume}[1]{\textbf{#1}}
\newcommand{\volumeyearpages}[3]{\textbf{#1} (#2), #3}
\newcommand{\volumeissueyearpages}[4]{\textbf{#1}:#2 (#3), #4}

\usepackage{amsmath,amssymb}
\usepackage{mathtools}
\newcommand{\eqdef}{\coloneqq}

\newcommand{\bC}{\mathbb{C}}
\newcommand{\bD}{\mathbb{D}}
\newcommand{\bZ}{\mathbb{Z}}
\newcommand{\bN}{\mathbb{N}}
\newcommand{\bNz}{\mathbb{N}_0}
\newcommand{\bR}{\mathbb{R}}
\newcommand{\bT}{\mathbb{T}}
\newcommand{\cA}{\mathcal{A}}
\newcommand{\cB}{\mathcal{B}}
\newcommand{\cD}{\mathcal{D}}
\newcommand{\cF}{\mathcal{F}}
\newcommand{\cR}{\mathcal{R}}

\newcommand{\cT}{\mathcal{T}}
\newcommand{\cU}{\mathcal{U}}
\newcommand{\fM}{\mathfrak{M}}
\newcommand{\fC}{\mathfrak{C}}
\newcommand{\fL}{\mathfrak{L}}
\newcommand{\econstant}{\operatorname{e}}
\newcommand{\imagunit}{\operatorname{i}}
\newcommand{\al}{\alpha}
\newcommand{\dif}{\mathrm{d}}
\newcommand{\la}{\lambda}
\newcommand{\tht}{\vartheta}
\newcommand{\Mat}{\mathcal{M}}
\renewcommand{\Re}{\operatorname{Re}}
\renewcommand{\Im}{\operatorname{Im}}

\newcommand{\clos}{\operatorname{clos}}
\newcommand{\lin}{\operatorname{span}}
\newcommand{\RB}{\operatorname{RB}}
\newcommand{\RBC}{\operatorname{RBC}}
\newcommand{\Gaussian}{\gamma}
\newcommand{\conjw}{\overline{w}}
\newcommand{\conjz}{\overline{z}}
\newcommand{\Berezin}{\operatorname{Ber}}
\newcommand{\Radialization}{\operatorname{Rad}}

\usepackage{colortbl}
\usepackage{xcolor}
\colorlet{lightblue}{blue!20}

\usepackage{amsthm}
\newtheorem{thm}{Theorem}[section]
\newtheorem{prop}[thm]{Proposition}
\newtheorem{conj}[thm]{Conjecture}
\newtheorem{lem}[thm]{Lemma}
\newtheorem{cor}[thm]{Corollary}
\theoremstyle{definition}
\newtheorem{example}[thm]{Example}
\newtheorem{defn}[thm]{Definition}
\newtheorem{remark}[thm]{Remark}

\author{Egor A. Maximenko,
Ana Mar\'{i}a Teller\'{i}a-Romero}
\title{Radial operators on polyanalytic\\Bargmann--Segal--Fock spaces}

\begin{document}%
\let\newpage\relax\maketitle

\begin{center}
Dedicated to Nikolai L. Vasilevski,
our guide in this area of mathematics,\\
on the occasion of his 70th birthday
\end{center}

\begin{abstract}
The paper considers bounded linear radial operators
on the polyanalytic Fock spaces $\mathcal{F}_n$
and on the true-polyanalytic Fock spaces $\mathcal{F}_{(n)}$.
The orthonormal basis of normalized complex Hermite polynomials
plays a crucial role in this study;
it can be obtained by the orthogonalization
of monomials in $z$ and $\overline{z}$.
First, using this basis,
we decompose the von Neumann algebra
of radial operators, acting in $\mathcal{F}_n$,
into the direct sum of some matrix algebras,
i.e. radial operators are represented as matrix sequences.
Secondly, we prove that the radial operators,
acting in $\mathcal{F}_{(n)}$,
are diagonal with respect to the basis
of the complex Hermite polynomials belonging to $\mathcal{F}_{(n)}$.
We also provide direct proofs
of the fundamental properties of $\mathcal{F}_n$
and an explicit description of the C*-algebra
generated by Toeplitz operators in $\mathcal{F}_{(n)}$,
whose generating symbols are radial, bounded,
and have finite limits at infinity.


\end{abstract}

\section{Introduction and main results}

The theory of bounded linear operators
in spaces of analytic functions
has been intensively developed since the 1980s.
In particular, the general theory of operators
on the Bargmann-Segal-Fock space
(for the sake of brevity, we will say just ``Fock space'')
is explained in the book of Zhu~\cite{Zhu2012book}.
Nevertheless, the complete understanding of the spectral properties
is achieved only for some special classes of operators,
in particular, for Toeplitz operators with generating symbols
invariant under some group actions,
see Vasilevski~\cite{Vasilevski2008book},
Grudsky, Quiroga-Barranco, and Vasilevski~\cite{GrudskyQuirogaVasilevski2006},
Dawson, \'{O}lafsson, and Quiroga-Barranco~\cite{DawsonOlafssonQuiroga2015}.
The simplest class of this type consists of Toeplitz operators
with bounded radial generating symbols.
Various properties of these operators
(boundedness, compactness, and eigenvalues)
have been studied by many authors,
see~\cite{KorenblumZhu1995,GrudskyVasilevski2002,Zorboska2003,Quiroga2016}.
The C*-algebra generated by such operators
was explicitly described in \cite{Suarez2008,GrudskyMaximenkoVasilevski2013}
for the nonweighted Bergman space,
in \cite{BauerHerreraVasilevski2014,HerreraVasilevskiMaximenko2015} for the weighted Bergman space,
and in \cite{EsmeralMaximenko2016} for the Fock space.
Loaiza and Lozano \cite{LoaizaLozano2013,LoaizaLozano2014}
studied radial Toeplitz operators in harmonic Bergman spaces.

The spaces of polyanalytic functions,
related with Landau levels,
have been used in mathematical physics since 1950s;
let us just mention a couple of recent papers:
\cite{HaimiHedenmalm2013,AliBagarelloGazeau2015}.
A connection of these spaces
with wavelet spaces and signal processing
is shown by Abreu~\cite{Abreu2010}
and Hutn\'{i}k~\cite{Hutnik2008,Hutnik2010}.
Various mathematicians contributed
to the rigorous mathematical theory of
square-integrable polyanalytic functions.
Our research is based on results and ideas from
\cite{Balk1991,Vasilevski2000,Shigekawa1987,AskourIntissarMouayn1997,AbreuFeichtinger2014}.

Hutn\'{i}k, Hutn\'{i}kov\'{a}, Ram\'{i}rez Ortega,
S\'{a}nchez-Nungaray, Loaiza, and other authors
\cite{HutnikHutnikova2015,RamirezSanchez2015,LoaizaRamirez2017,SanchezGonzalezLopezArroyo2018,HutnikMaximenkoMiskova2016} studied vertical and angular Toeplitz operators in polyanalytic
and true-polyanalytic spaces, Bergman and Fock.
In particular, vertical Toeplitz operators in the
$n$-analytic Bergman space over the upper half-plane
are represented in \cite{RamirezSanchez2015}
as $n\times n$ matrices
whose entries are continuous functions on $(0,+\infty)$,
with some additional properties at $0$ and $+\infty$.

Recently, Rozenblum and Vasilevski
\cite{RozenblumVasilevski2018}
investigated Toeplitz operators with distributional symbols
and showed that Toeplitz operators
in true-polyanalytic Fock spaces
are equivalent to some Toeplitz operators
with distributional symbols in the analytic Fock space.

In this paper, we analyze radial operators in Fock spaces
of polyanalytic or true-polyanalytic functions.
We denote by $\mu$ the Lebesque measure on the complex plane
and by $\gamma$ the Gaussian measure on the complex plane:
\[
\dif\gamma(z)=\frac{1}{\pi}\econstant^{-|z|^2}\dif\mu(z).
\]
In what follows, we work with the space $L^2(\bC,\gamma)$
and its subspaces, and denote its norm by $\|\cdot\|$.
A very useful orthonormal basis in $L^2(\bC,\gamma)$
is formed by complex Hermite polynomials $b_{j,k}$,
$j,k\in\bNz\eqdef\{0,1,2,\ldots\}$;
see Section~\ref{sec:basis}.

Given $n$ in $\bN\eqdef\{1,2,\ldots\}$,
let $\cF_n$ be the subspace of $L^2(\bC,\gamma)$
consisting of all $n$-analytic functions belonging to $L^2(\bC,\gamma)$.
It is known that $\cF_n$ is a closed subspace of $L^2(\bC,\gamma)$;
moreover, it is a RKHS (reproducing kernel Hilbert space).
We denote by $\cF_{(n)}$ the orthogonal complement of
$\cF_{n-1}$ in $\cF_{n}$.

For every $\tau$ in $\bT\eqdef\{z\in\bC\colon\ |z|=1\}$,
let $R_{n,\tau}$ be the rotation operator acting in $\cF_n$ by the rule
\[
(R_{n,\tau}f)(z)\eqdef f(\tau^{-1}z).
\]
The family $(R_{n,\tau})_{\tau\in\bT}$ is a
unitary representation of the group $\bT$ in the space $\cF_n$.
We denote by $\cR_n$ the commutant of $\{R_{n,\tau}\colon\ \tau\in\bT\}$ in $\cB(\cF_n)$,
i.e. the von Neumann algebra that consists of all bounded linear operators
acting in $\cF_n$ that commute with $R_{n,\tau}$
for every $\tau$ in $\bT$.
In other words, the elements of $\cR_n$ are the operators
intertwining the representation $(R_{n,\tau})_{\tau\in\bT}$
of the group $\bT$.
The elements of $\cR_n$ are called
\emph{radial} operators in $\cF_n$.

In a similar manner,
we denote by $R_{(n),\tau}$ the rotation operators acting in $\cF_{(n)}$
and by $\cR_{(n)}$ the von Neumann algebra
of radial operators in $\cF_{(n)}$.

The principal tool in the study of $\cR_n$
is the following orthogonal decomposition of $\cF_n$:
\begin{equation}\label{eq:cFn_decomposition_into_diagonal_subspaces}
\cF_n=\bigoplus_{d=-n+1}\cD_{d,\min\{n,n+d\}}.
\end{equation}
Here the ``truncated diagonal subspaces'' $\cD_{d,m}$ are defined
as the
linear spans of $b_{j,k}$ with $j-k=d$
and $0\le j,k<m$.
Another description of $\cD_{d,m}$ is given in Proposition~\ref{prop:truncated_diagonals_in_polar_coordinates}.

The main results of this paper are explicit decompositions
of the von Neumann algebras $\cR_n$ and $\cR_{(n)}$
into direct sums of factors.
The symbol $\cong$ means that the algebras are
isometrically isomorphic.

\begin{thm}\label{thm:radial_polyanalytic_Fock}
Let $n\in\bN$.
Then $\cR_n$ consists of all operators belonging to $\cB(\cF_n)$
that act invariantly on the subspaces
$\cD_{d,\min\{n,n+d\}}$, for $d\ge-n+1$.
Furthermore,
\[
\cR_n
\cong\bigoplus_{d=-n+1}^\infty \cB(\cD_{d,\min\{n,n+d\}})
\cong \bigoplus_{d=-n+1}^\infty \Mat_{\min\{n,n+d\}}.
\]
\end{thm}

\begin{thm}\label{thm:radial_true_polyanalytic_Fock}
Let $n\in\bN$.
Then $\cR_{(n)}$ consists of all operators
belonging to $\cB(\cF_{(n)})$
that are diagonal with respect to the orthonormal basis
$(b_{p,n-1})_{p=0}^\infty$.
Furthermore,
\[
\cR_{(n)}\cong\ell^\infty(\bNz).
\]
\end{thm}

In particular, Theorems~\ref{thm:radial_polyanalytic_Fock}
and \ref{thm:radial_true_polyanalytic_Fock}
imply that the algebra $\cR_n$ is noncommutative for $n\ge2$,
whereas $\cR_{(n)}$ is commutative for every $n$ in $\bN$.

In Section~\ref{sec:basis} we recall the main properties
of the complex Hermite polynomials $b_{p,q}$.
In Section~\ref{sec:spaces} we give direct proofs
of the principal properties of the spaces $\cF_n$ and $\cF_{(n)}$.
Section~\ref{sec:unitary_representations}
contains some general remarks
about unitary representations in RKHS,
given by changes of variables.
Section~\ref{sec:radial} deals with radial operators,
describes the von Neumann algebra of radial operators
in $L^2(\bC,\gamma)$,
and proves Theorems~\ref{thm:radial_polyanalytic_Fock}
and \ref{thm:radial_true_polyanalytic_Fock}.
Finally, in Section~\ref{sec:Toeplitz}
we make some simple observations
about Toeplitz operators generated by bounded radial functions
and acting in the spaces $\cF_n$ and $\cF_{(n)}$.

Another natural method to prove
\eqref{eq:cFn_decomposition_into_diagonal_subspaces}
and Theorems~\ref{thm:radial_polyanalytic_Fock},
\ref{thm:radial_true_polyanalytic_Fock}
is to represent $L^2(\bC,\Gaussian)$ as a tensor product
$L^2(\bT,\dif\mu_T)\otimes L^2([0,+\infty),
\econstant^{-r^2}\,2r\,\dif{}r)$
and to apply the Fourier transform of the group $\bT$.
We prefer to work with the canonical basis
because this method seems more elementary.

Comparing our Theorem~\ref{thm:radial_polyanalytic_Fock}
with the main results of \cite{RamirezSanchez2015,LoaizaRamirez2017,SanchezGonzalezLopezArroyo2018},
we would like to point out three differences.
\begin{enumerate}
\item We study the von Neumann algebra $\cR_n$ of
\emph{all radial operators},
instead of C*-algebras generated by Toeplitz operators with radial symbols (such C*-algebras can be objects of study in a future).
\item The dual group of $\bT$ is the discrete group $\bZ$,
therefore matrix sequences appear instead of matrix functions.
\item In \cite{RamirezSanchez2015,LoaizaRamirez2017,SanchezGonzalezLopezArroyo2018},
all matrices have the same order $n$,
whereas in our Theorem~\ref{thm:radial_polyanalytic_Fock}
the matrices have orders $1,2,\ldots,n-1,n,n,\ldots$.
\end{enumerate}

\section{Complex Hermite polynomials}
\label{sec:basis}

Most results of Sections~\ref{sec:basis} and \ref{sec:spaces}
are well known to experts
\cite{Balk1991,Vasilevski2000,AbreuFeichtinger2014}.
Nevertheless, our proofs are more direct
than the ideas found in the bibliography.

Given a function $f\colon\bC\to\bC$,
continuously differentiable in the $\bR^2$-sense,
we define $A^\dagger f$ and $\overline{A}^\dagger f$ by
\begin{align*}
A^\dagger f
&=\left(\conjz - \frac{\partial}{\partial z}\right)f
=-\econstant^{z\,\conjz}
\frac{\partial}{\partial z}
\left(\econstant^{-z\,\conjz} f\right),
\\
\overline{A}^\dagger f
&=\left(z - \frac{\partial}{\partial\conjz}\right)f
=-\econstant^{z\,\conjz}
\frac{\partial}{\partial\conjz}
\left(\econstant^{-z\,\conjz} f\right).
\end{align*}
The operators $A^\dagger$ and $\overline{A}^\dagger$
are known as (nonnormalized) creation operators
with respect to $\conjz$ and $z$, respectively.
For every $p,q$ in $\bNz$, denote by $m_{p,q}$
the monomial function $m_{p,q}(z)\eqdef z^p\, \overline{z}^q$.
Following Shigekawa \cite[Section~7]{Shigekawa1987}
we define the \emph{normalized complex Hermite polynomials} as
\begin{equation}\label{eq:b_Shigekawa}
b_{p,q} \eqdef
\frac{1}{\sqrt{p!\,q!}}
(A^\dagger)^q 
(\overline{A}^\dagger)^p
m_{0,0}\qquad(p,q\in\bNz).
\end{equation}
Notice that \cite{Shigekawa1987} defines
complex Hermite polynomials without the
factor $\frac{1}{\sqrt{p!\,q!}}$.
These polynomials appear also in Balk~\cite[Section~6.3]{Balk1991}.
Let us show explicitly some of them:
\[
\begin{array}{ccc}
b_{0,0}(z) = 1, &
b_{0,1}(z) = \overline{z}, &
b_{0,2}(z) = \frac{1}{\sqrt{2}}\overline{z}^2,
\\[1ex]
b_{1,0}(z) = z, &
b_{1,1}(z) = |z|^2-1, &
b_{1,2}(z) = \frac{1}{\sqrt{2}}\,\overline{z}(|z|^2-2),
\\[1ex]
b_{2,0}(z) = \frac{1}{\sqrt{2}}\,z^2, &
b_{2,1}(z) = \frac{1}{\sqrt{2}}\,z (|z|^2-2), &
b_{2,2}(z) = \frac{1}{2}(|z|^2 - 4|z|^2 + 2).
\end{array}
\]
For every $p,\alpha$ in $\bNz$, we denote by $L_p^{(\alpha)}$
the associated Laguerre polynomial.
Recall the Rodrigues formula,
the explicit expression,
and the orthogonality relation for these polynomials:
\begin{align}
\label{eq:generalized_Laguerre_Rodriguez}
L_n^{(\al)}(x)
&=\frac{x^{-\al}\econstant^x}{n!}
\frac{\dif^n}{\dif{}x^n}\bigl(\econstant^{-x}\,x^{n+\al}\bigr),
\\
\label{eq:generalized_Laguerre_explicit}
L_n^{(\alpha)}(x)
&=\sum_{k=0}^n (-1)^k \binom{n+\al}{n-k}\frac{x^k}{k!},
\\
\label{eq:Laguerre_orthogonality}
\int_0^{+\infty} L_n^{(\alpha)}(x) & L_m^{(\alpha)}(x)
\,x^\alpha \econstant^{-x}\,\dif{}x
=\frac{(n+\alpha)!}{n!}\,\delta_{m,n}.
\end{align}

\begin{lem}\label{lem:derivative_from_Rodriguez_xy}
Let $n,\alpha\in\bN$. Then
\begin{equation}\label{eq:derivative_from_Rodriguez_xy}
\econstant^{xy}\frac{\partial^n}{\partial{}x^n}
\bigl(\econstant^{-xy}\,x^{n+\al}\bigr)
=n!\,x^\al L_n^{(\al)}(xy).
\end{equation}
\end{lem}

\begin{proof}
Apply Rodrigues formula~\eqref{eq:generalized_Laguerre_Rodriguez}
and the chain rule:
\[
\frac{\partial^n}{\partial{}x^n}
\bigl(\econstant^{-xy}(xy)^{n+\al}\bigr)
=n!\,\econstant^{-xy}(xy)^\al L_n^{(\al)}(xy)\,y^n.
\]
Canceling the factor $y^{n+\al}$ in both sides yields~\eqref{eq:derivative_from_Rodriguez_xy}.
\end{proof}

\begin{prop}
For every $p,q$ in $\bNz$,
\begin{equation}\label{eq:b_Laguerre}
b_{p,q}(z)
=
\begin{cases}
\sqrt{\frac{q!}{p!}} (-1)^{q} \,z^{p-q} L_q^{(p-q)}(|z|^2), & \text{if}\ p\ge q;\\[1ex]
\sqrt{\frac{p!}{q!}} (-1)^{p} \,\overline{z}^{q-p} L_p^{(q-p)}(|z|^2), & \text{if}\ p\le q.
\end{cases}
\end{equation}
In other words,
\begin{equation}\label{eq:b_via_m}
b_{p,q}
=\sqrt{\frac{\min\{p,q\}!}{\max\{p,q\}!}}
\sum_{s=0}^{\min\{p,q\}}
\binom{\max\{p,q\}}{s}\,\frac{(-1)^s}{(\min\{p,q\}-s)!}
m_{p-s,q-s}.
\end{equation}
\end{prop}

\begin{proof}
Let $p,q\in\bNz$, $p\ge q$.
Notice that $\frac{\partial}{\partial z}|z|^2
=\frac{\partial}{\partial z}(z\,\conjz)=\conjz$.
By~\eqref{eq:b_Shigekawa} and \eqref{eq:derivative_from_Rodriguez_xy},
\begin{align*}
b_{p,q}(z)
&=\frac{(-1)^{p+q}}{\sqrt{p!\,q!}}
\econstant^{z\conjz}
\frac{\partial^q}{\partial{}z^q}
\frac{\partial^p}{\partial{}\conjz^p}\,
\econstant^{-z\conjz}
\\
&=
\frac{(-1)^q}{\sqrt{p!\,q!}}
\econstant^{z\conjz}
\frac{\partial^q}{\partial z^q}
\bigl(\econstant^{-z\conjz} z^p\bigr)
=\sqrt{\frac{q!}{p!}}\,(-1)^q
z^{p-q} L_q^{(p-q)}(z\,\conjz).
\end{align*}
In the case
when
$p\le q$,
we first notice that the operators
$A^\dagger$ and $\overline{A}^\dagger$ commute
on the space of polynomial functions.
Reasoning as above, but swapping the roles of $z$ and $\conjz$, we arrive at the second case of~\eqref{eq:b_Laguerre}.
Finally, with the help of \eqref{eq:generalized_Laguerre_explicit},
we pass from \eqref{eq:b_Laguerre} to \eqref{eq:b_via_m}.
Formula \eqref{eq:b_via_m} can also be derived directly
from~\eqref{eq:b_Shigekawa},
by applying mathematical
induction
and working with binomial coefficients.
\end{proof}

Denote by $\ell_m^{(\alpha)}$ the \emph{normalized Laguerre function}:
\begin{equation}\label{eq:laguerre_function}
\ell_m^{(\alpha)}(t)\eqdef
\sqrt{\frac{m!}{(m+\alpha)!}}\,
t^{\alpha/2} \econstant^{-t/2} L_m^{(\alpha)}(t)
\qquad(m,\alpha\in\bNz).
\end{equation}

\begin{cor}\label{cor:b_via_Laguerre_function}
For every $p,q$ in $\bNz$,
\begin{equation}\label{eq:b_via_Laguerre_function}
b_{p,q}(r\tau)
=(-1)^{\min\{p,q\}}\tau^{p-q}
\econstant^{r^2/2}\ell_{\min\{p,q\}}^{(|p-q|)}(r^2)
\qquad(r\ge0,\ \tau\in\bT).
\end{equation}
\end{cor}

It is convenient to treat the family
$(m_{p,q})_{p,q\in\bNz}$ as an infinite table,
and to think in terms of its columns or diagonals
(parallel to the main diagonal).
Given $d$ in $\bZ$ and $n$ in $\bNz$,
let $\cD_{d,n}$ be the subspace of $L^2(\bC,\gamma)$
generated by the first $n$ monomials in the diagonal with index $d$:
\[
\cD_{d,n} \eqdef \lin\{m_{p,q}\colon\ p,q\in\bNz,\ \min\{p,q\}<n,\ p-q=d\}.
\]

\begin{prop}\label{prop:basis_in_L2}
The family $(b_{p,q})_{p,q\in\bNz}$ is
an orthonormal basis of $L^2(\bC,\gamma)$.
This family can be obtained from
$(m_{p,q})_{p,q=0}^\infty$
by applying the Gram--Schmidt orthogonalization.
\end{prop}

\begin{proof}
1. The orthonormality is easy to verify
by passing to polar coordinates
and using \eqref{eq:b_Laguerre}
with the orthogonality relation
\eqref{eq:Laguerre_orthogonality}.

2. The formula~\eqref{eq:b_via_m} tells us
that the functions $b_{p,q}$ are linear combinations
of $m_{p-s,q-s}$ with $0\le s\le\min\{p,q\}$.
Inverting these formulas,
$m_{p,q}$ results a linear combination
of $b_{p-s,q-s}$ with $0\le s\le\min\{p,q\}$.
So, for every $d$ in $\bZ$ and every $n$ in $\bNz$,
\begin{equation}\label{eq:truncated_diagonal_through_b}
\cD_{d,n}=\lin\{b_{p,q}\colon\ p,q\in\bNz,\ \min\{p,q\}<n,\ p-q=d\}.
\end{equation}
Jointly with the orthonormality of $(b_{p,q})_{p,q=0}^\infty$,
this means that the family $(b_{p,q})_{p,q=0}^\infty$
is obtained from $(m_{p,q})_{p,q=0}^\infty$
by applying the orthogonalization in each diagonal.

3. Due to 2, it is sufficient to prove that
the polynomials in $z$ and $\overline{z}$ form a dense subset of $L^2(\bC,\gamma)$.
Notice that the set of polynomial functions in $z$ and $\overline{z}$
coincides with the set of polynomial functions in $\Re(z)$ and $\Im(z)$.
Suppose that $f\in L^2(\bC,\gamma)$
and $f$ is orthogonal to the polynomials $\Re(z)^j \Im(z)^k$ for all $j,k$ in $\bNz$.
Denote by $g$ the function $g(x,y)=f(x+\imagunit y)\econstant^{-x^2-y^2}$
and consider its Fourier transform:
\begin{align*}
\widehat{g}(u,v)
&=\int_{\bR^2} \econstant^{-2\pi\imagunit(xu+yv)}\,
f(x+\imagunit y)\econstant^{-x^2-y^2}\,\dif{}x\,\dif{}y
\\[0.5ex]
&=\sum_{j=0}^\infty\sum_{k=0}^\infty
\frac{(-2\pi\imagunit u)^j (-2\pi\imagunit v)^k}{j!\,k!}
\int_{\bR^2} x^j y^k f(x+\imagunit y) \econstant^{-x^2-y^2}\,\dif{}x\,\dif{}y
=0.
\end{align*}
By the injective property of the Fourier transform,
we conclude that $g$ vanishes a.e.
As a consequence, $f$ also vanishes a.e.
\end{proof}

\begin{remark}\label{rem:products_b_m}
The second part of the proof of Proposition~\ref{prop:basis_in_L2} implies that
for every $d$ in $\bZ$, every $q\ge\max\{0,-d\}$ every $k$ in $\bZ$
with $\max\{0,-d\}\le k\le q$,
\begin{equation}\label{eq:product_b_m}
\langle m_{d+k,k},b_{d+q,q}\rangle
=
\begin{cases}
\sqrt{q!\,(d+q)!}, & k=q; \\[0.5ex]
0, & k<q.
\end{cases}
\end{equation}
\end{remark}

Formula~\eqref{eq:truncated_diagonal_through_b}
means that the first $n$ elements in the diagonal $d$
of the table $(b_{p,q})_{p,q\in\bNz}$
generate the same subspace
as the first $n$ elements in the diagonal $d$
of the table $(m_{p,q})_{p,q\in\bNz}$.
For example,
\begin{align*}
\cD_{-1,3}
&=\lin\{m_{0,1},m_{1,2},m_{2,3}\}=\lin\{b_{0,1},b_{1,2},b_{2,3}\},\\
\cD_{2,2}
&=\lin\{m_{2,0},m_{3,1}\}=\lin\{b_{2,0},b_{3,1}\}.
\end{align*}
In the following tables we show generators of $\cD_{2,2}$ (green) and $\cD_{-1,3}$ (blue).
\[
\begin{array}{cccccc}
m_{0,0} & \cellcolor{lightblue} m_{0,1} & m_{0,2} & m_{0,3} & m_{0,4} & \ddots \\
m_{1,0} & m_{1,1} & \cellcolor{lightblue} m_{1,2} & m_{1,3} & m_{1,4} & \ddots \\
\cellcolor{green} m_{2,0} & m_{2,1} & m_{2,2} & \cellcolor{lightblue} m_{2,3} & m_{2,4} & \ddots \\
m_{3,0} & \cellcolor{green} m_{3,1} & m_{3,2} & m_{3,3} & m_{3,4} & \ddots \\
m_{4,0} & m_{4,1} & m_{4,2} & m_{4,3} & m_{4,4} & \ddots \\
\ddots & \ddots & \ddots & \ddots & \ddots & \ddots
\end{array}
\qquad\quad
\begin{array}{cccccc}
b_{0,0} & \cellcolor{lightblue} b_{0,1} & b_{0,2} & b_{0,3} & b_{0,4} & \ddots \\
b_{1,0} & b_{1,1} & \cellcolor{lightblue} b_{1,2} & b_{1,3} & b_{1,4} & \ddots \\
\cellcolor{green} b_{2,0} & b_{2,1} & b_{2,2} & \cellcolor{lightblue} b_{2,3} & b_{2,4} & \ddots \\
b_{3,0} & \cellcolor{green} b_{3,1} & b_{3,2} & b_{3,3} & b_{3,4} & \ddots \\
b_{4,0} & b_{4,1} & b_{4,2} & b_{4,3} & b_{4,4} & \ddots \\
\ddots & \ddots & \ddots & \ddots & \ddots & \ddots
\end{array}
\]
Given $d$ in $\bZ$, we denote by $\cD_d$
the closure of the subspace of $L^2(\bC,\gamma)$
generated by the monomials $m_{p,q}$, where $p-q=d$:
\[
\cD_d \eqdef \clos\bigl(\lin\{m_{p,q}\colon\ p,q\in\bNz,\ p-q=d\}\bigr).
\]
Proposition~\ref{prop:basis_in_L2}
implies the following properties
of the ``diagonal subspaces'' $\cD_d$, $d\in\bZ$.

\begin{cor}\label{cor:basis_of_D}
The sequence $(b_{q+d,q})_{q=\max\{0,-d\}}^\infty$
is an orthonormal basis of $\cD_d$.
\end{cor}

\begin{cor}\label{cor:polar_description_of_D}
The space $\cD_d$ consists of all functions of the form
\begin{equation}\label{eq:polar_description_of_D}
f(r\tau)=\tau^{d} h(r^2)\quad
(r\ge0,\ \tau\in\bT),\quad
\text{where}\quad
h\in L^2([0,+\infty),\econstant^{-x}\,\dif{}x).
\end{equation}
Moreover, $\|f\|=\|h\|_{L^2([0,+\infty),\econstant^{-x}\,\dif{}x)}$.
\end{cor}

\begin{cor}\label{cor:decomposition_L2_into_diagonals}
The space $L^2(\bC,\Gaussian)$
is the orthogonal sum of the subspaces $\cD_d$:
\begin{equation}\label{eq:decomposition_L2_into_diagonal_subspaces}
L^2(\bC,\Gaussian)=\bigoplus_{d\in\bZ}\cD_d.
\end{equation}
\end{cor}

Here we show the generators of $\cD_1$ (green) and $\cD_{-2}$ (blue): 
\[
\begin{array}{ccccc}
m_{0,0} & m_{0,1} & \cellcolor{lightblue} m_{0,2} & m_{0,3} & \ddots \\
\cellcolor{green} m_{1,0} & m_{1,1} & m_{1,2} & \cellcolor{lightblue}  m_{1,3} & \ddots \\
m_{2,0} & \cellcolor{green}  m_{2,1} & m_{2,2} & m_{2,3} & \cellcolor{lightblue} \ddots \\
m_{3,0} & m_{3,1} & \cellcolor{green} m_{3,2} & m_{3,3} & \ddots \\
\ddots & \ddots & \ddots & \cellcolor{green} \ddots & \ddots
\end{array}
\qquad\quad
\begin{array}{ccccc}
b_{0,0} & b_{0,1} & \cellcolor{lightblue} b_{0,2} & b_{0,3} & \ddots \\
\cellcolor{green} b_{1,0} & b_{1,1} & b_{1,2} & \cellcolor{lightblue} b_{1,3} & \ddots \\
b_{2,0} & \cellcolor{green} b_{2,1} & b_{2,2} & b_{2,3} & \cellcolor{lightblue} \ddots \\
b_{3,0} & b_{3,1} & \cellcolor{green} b_{3,2} & b_{3,3} & \ddots \\
\ddots & \ddots & \ddots & \cellcolor{green} \ddots & \ddots
\end{array}
\]

\section{Bargmann--Segal--Fock spaces of polyanalytic functions}
\label{sec:spaces}

Fix $n$ in $\bN$.
Let $\cF_n$ be the space of $n$-polyanalytic functions
belonging to $L^2(\bC,\gamma)$,
and $\cF_{(n)}$ be the true-$n$-polyanalytic Fock space
defined in \cite{Vasilevski2000} by
\[
\cF_{(n)}\eqdef \{f\in \cF_n\colon\ f\perp \cF_{n-1}\}.
\]

\begin{prop}\label{prop:evaluation_functionals_on_Fn_are_bounded}
Let $R>0$.
Then there exists a number $C_{n,R}>0$
such that for every $f$ in $\cF_n$
and every $z$ in $\bC$ with $|z|\le R$,
\begin{equation}\label{eq:evaluation_functionals_on_Fn_are_bounded}
|f(z)|\le C_{n,R} \|f\|.
\end{equation}
\end{prop}

\begin{proof}
Let $P_n$ be the polynomial in one variable of degree $\le n-1$
such that
\begin{equation}\label{eq:polynomial_orthogonal_to_positive_powers}
\int_0^1 P_n(x) x^j\,\dif{}x=\delta_{j,0}
\qquad(j\in\{0,\ldots,n-1\}).
\end{equation}
The existence and uniqueness of such a polynomial
follows from the invertibility of the Hilbert matrix
$\bigl[1/(j+k+1)\bigr]_{j,k=0}^{n-1}$.
Put
\[
C_{n,R}
\eqdef\left(\max_{x\in[0,1]}|P_n(x)|\right)\,
\left(\frac{1}{\pi}
\int_{(R+1)\bD} \econstant^{|w|^2}\,\dif\mu(w)\right)^{1/2}.
\]
Let $f\in\cF_n$ and $z\in\bC$, with $|z|\le R$.
It is known \cite[Section~1.1]{Balk1991}
that $f$ can be expanded into a uniformly convergent series
of the form
\[
f(w)=\sum_{j=0}^\infty \sum_{k=0}^{n-1} \alpha_{j,k}
(w-z)^j (\conjw-\conjz)^k,
\]
where $\alpha_{j,k}$ are some complex numbers.
Using the change of variables $w=z+r\econstant^{\imagunit\tht}$
and the property \eqref{eq:polynomial_orthogonal_to_positive_powers},
we obtain the following version of
the mean value property of polyanalytic functions:
\begin{equation}\label{eq:polyanalytic_mean_value}
f(z)=\frac{1}{\pi} \int_{z+\bD}f(w) P_n(|w-z|^2)\,\dif\mu(w).
\end{equation}
After that, estimating $|P_n|$ by its maximum value,
multiplying and dividing by $\econstant^{|w|^2/2}$,
and applying the Schwarz inequality,
we arrive at \eqref{eq:evaluation_functionals_on_Fn_are_bounded}.
\end{proof}

\begin{remark}
The constant $C_{n,R}$, found in the proof of Proposition~\ref{prop:evaluation_functionals_on_Fn_are_bounded},
is not optimal.
The exact upper bound for the evaluation functionals in $\cF_n$
is given in Corollary~\ref{cor:evaluation_Fn_upper_bound}.
\end{remark}

\begin{prop}
$\cF_n$ is a RKHS.
\end{prop}

\begin{proof}
Let $(g_n)_{n\in\bN}$ be a Cauchy sequence in $\cF_n$.
By Proposition~\ref{prop:evaluation_functionals_on_Fn_are_bounded},
this sequence converges pointwise on $\bC$
and uniformly on compacts to a function $f$.
By \cite[Corollary~1.8]{Balk1991},
the function $f$ is $n$-analytic.
On the other hand, let $h$ be the limit of the sequence
$(g_n)_{n\in\bN}$ in $L^2(\bC,\Gaussian)$.
Then for every compact $K$ in $\bC$,
the sequence of the restrictions $g_n|_K$
converges in the $L^2(K,\Gaussian)$-norm
simultaneously to $f|_K$ and to $h|_K$.
Therefore $h$ coincides with $f$ a.e.
and $f\in L^2(\bC,\Gaussian)$, i.e. $f\in\cF_n$.
So, $\cF_n$ is a Hilbert space.
The boundedness of the evaluation functionals
is established in Proposition~\ref{prop:evaluation_functionals_on_Fn_are_bounded}.
\end{proof}

\begin{prop}\label{prop:basis_in_Fn}
The family $(b_{p,q})_{p\in\bNz,q<n}$
is an orthonormal basis of $\cF_{n}$.
\end{prop}

\begin{proof}
We already know that this family is contained in $\cF_n$
and is orthonormal.
Let us verify the total property.
Our reasoning uses ideas of
Ramazanov \cite[proof of Theorem~2]{Ramazanov1999}.

Suppose that $f\in\cF_n$ and $\langle f,b_{p,q}\rangle=0$
for every $p\in\bNz$, $q<n$.
We have to show that $f=0$.
By the decomposition of polyanalytic functions
\cite[Section~1.1]{Balk1991},
there exists a family of numbers $(\alpha_{j,k})_{j\in\bNz,k<n}$ such that
\[
f(z)=\sum_{k=0}^{n-1}\sum_{j=0}^\infty
\alpha_{j,k} m_{j,k}(z),
\]
where each of the inner series converges pointwise on $\bC$
and uniformly on compacts.
For every $\nu$ in $\bNz$,
we denote by $S_\nu$ the partial sum
$S_\nu\eqdef \sum_{k=0}^{n-1}\sum_{j=0}^\nu \alpha_{j,k}m_{j,k}$.
Given $r>0$, the sequence $(S_\nu)_{\nu\in\bNz}$
converges to $f$ uniformly on $r\bD$.
For every $p,q$ in $\bNz$ with $q<n$,
using the orthogonality on $r\bD$
between $b_{p,q}$ and $m_{j,k}$
with $j-k\ne p-q$, we obtain
\[
\int_{r\bD} f\,\overline{b_{p,q}}\,\dif\Gaussian
=\lim_{\nu\to\infty}
\int_{r\bD} S_\nu\,\overline{b_{p,q}}\,\dif\Gaussian
=\sum_{k=0}^{n-1} \alpha_{k+p-q,k}
\int_{r\bD} m_{k+p-q,k} \overline{b_{p,q}}\,\dif\Gaussian.
\]
The functions $f\,\overline{b_{p,q}}$
and $m_{k+p-q,k}\,\overline{b_{p,q}}$
are integrable on $\bC$
with respect to the measure $\Gaussian$.
Therefore their integrals over $\bC$ are the limits
of the corresponding integrals over $r\bD$, as $r$ tends to infinity.
Since $\langle f,b_{p,q}\rangle=0$,
the coefficients $\alpha_{j,k}$
must satisfy the following infinite system
of homogeneous linear equations:
\begin{equation}\label{eq:system_of_equations_on_alpha}
\sum_{k=0}^{n-1} 
\langle m_{k+p-q,k}, b_{p,q}\rangle
\alpha_{k+p-q,k}
=0
\qquad(p\in\bNz,\ 0\le q<n).
\end{equation}
Now we fix $d>-n$ and restrict ourselves to the equations
\eqref{eq:system_of_equations_on_alpha} with $p-q=d$,
which yields an $s\times s$ system represented by the matrix $M_d$,
where $s=\min\{n,n+d\}$, and
\[
M_d\eqdef\left[
\langle m_{d+k,k},b_{d+q,q}
\rangle
\right]_{q,k=\max\{0,-d\}}^{n-1}.
\]
By \eqref{eq:product_b_m}, $M_d$ is an upper triangular matrix
with nonzero diagonal entries, hence $M_d$ is invertible.
So, all coefficients $\alpha_{j,k}$ are zero.
\end{proof}

\begin{cor}\label{cor:basis_in_true_polyanalytic}
$\cF_{(n)}$ is a RKHS,
and the sequence $(b_{p,n-1})_{p\in\bNz}$
is an orthonormal basis of $\cF_{(n)}$.
\end{cor}

We denote by $P_n$ and $P_{(n)}$ the orthogonal projections
acting in $L^2(\bC,\Gaussian)$,
whose images are $\cF_n$ and $\cF_{(n)}$, respectively.
They can be explicitly defined in terms
of the corresponding reproducing kernels:
\[
(P_n f)(z)=\langle f,K_{n,z} \rangle,\qquad
(P_{(n)}f)(z)=\langle f,K_{(n),z} \rangle.
\]

\begin{cor}
If $f\in\cF_n$, then
\[
f=\sum_{j=0}^\infty\sum_{k=0}^{n-1}
\langle f,b_{j,k}\rangle b_{j,k},
\]
where the series converges in the $L^2(\bC,\Gaussian)$-norm
and uniformly on the compacts.
In particular,
if $f\in\cF_{(n)}$, then
\begin{equation}\label{eq:f_true_polyanalytic_decomposition}
f=\sum_{j=0}^\infty
\langle f,b_{j,n-1}\rangle b_{j,n-1}.
\end{equation}
\end{cor}

For example, $(b_{p,2})_{p\in\bNz}$ is an orthonormal basis of $\cF_{(3)}$,
and $(b_{p,q})_{p\in\bNz,q<3}$ is an orthonormal basis of $\cF_3$:
\[
\begin{array}{ccccc}
b_{0,0} & b_{0,1} & \cellcolor{green} b_{0,2} & b_{0,3} & \ldots \\
b_{1,0} & b_{1,1} & \cellcolor{green} b_{1,2} & b_{1,3} & \ldots \\
b_{2,0} & b_{2,1} & \cellcolor{green} b_{2,2} & b_{2,3} & \ldots \\
b_{3,0} & b_{3,1} & \cellcolor{green} b_{3,2} & b_{3,3} & \ldots \\
\vdots & \vdots & \cellcolor{green} \vdots & \vdots & \ddots
\end{array}
\qquad\qquad
\begin{array}{ccccc}
\cellcolor{green} b_{0,0} & \cellcolor{green} b_{0,1} & \cellcolor{green} b_{0,2} & b_{0,3} & \ldots \\
\cellcolor{green} b_{1,0} & \cellcolor{green} b_{1,1} & \cellcolor{green} b_{1,2} & b_{1,3} & \ldots \\
\cellcolor{green} b_{2,0} & \cellcolor{green} b_{2,1} & \cellcolor{green} b_{2,2} & b_{2,3} & \ldots \\
\cellcolor{green} b_{3,0} & \cellcolor{green} b_{3,1} & \cellcolor{green} b_{3,2} & b_{3,3} & \ldots \\
\cellcolor{green} \vdots & \cellcolor{green} \vdots & \cellcolor{green} \vdots & \vdots & \ddots
\end{array}
\]
Using Proposition~\ref{prop:basis_in_Fn},
Corollary~\ref{cor:basis_of_D},
and formula~\eqref{eq:truncated_diagonal_through_b} gives
\begin{equation}\label{eq:truncated_diagonal_as_intersection}
\cD_d\cap\cF_n
=\begin{cases}
\cD_{d,\min\{n,n+d\}}, & d\ge -n+1; \\
\{0\}, & d<-n+1.
\end{cases}
\end{equation}
Here is a description of the subspaces $\cD_{d,m}$
in terms of the polar coordinates.

\begin{prop}\label{prop:truncated_diagonals_in_polar_coordinates}
For every $m$ in $\bNz$ and every $d$ in $\bZ$ with $d\ge-m+1$,
the space $\cD_{d,m}$ consists of all functions of the form
\[
f(r\tau)
= \tau^d r^{|d|} Q(r^2)\qquad
(r\ge0,\ \tau\in\bT),
\]
where $Q$ is a polynomial of degree $\le m-1$.
Moreover,
\[
\|f\|=\|Q\|_{L^2([0,+\infty),x^{|d|}\econstant^{-x}\,\dif{}x)}.
\]
\end{prop}

\begin{proof}
Apply formula~\eqref{eq:truncated_diagonal_through_b}
and the orthonormality of the polynomials $L_k^{(|d|)}$
in the space $L^2([0,+\infty),x^{|d|}\econstant^{-x}\,\dif{}x)$.
\end{proof}

The decomposition of $\cF_n$ into a direct sum of ``truncated diagonals'' shown below follows
from Proposition~\ref{prop:basis_in_Fn}
and plays a crucial role in the study of radial operators.

\begin{prop}\label{prop:decomposition_Fn_into_D}
\begin{equation}\label{eq:decomposition_Fn_into_D}
\cF_n = \bigoplus_{d=-n+1}^\infty \cD_{d,\min\{n,n+d\}}.
\end{equation}
\end{prop}

Let us illustrate Proposition~\ref{prop:decomposition_Fn_into_D} for $n=3$ with a table (we have marked in different shades of blue
the basic functions that generate each truncated diagonal):
\[
\begin{array}{ccccc}
\cellcolor{blue!32} b_{0,0} &
\cellcolor{blue!24} b_{0,1} &
\cellcolor{blue!16} b_{0,2} &
b_{0,3} & \ldots \\
\cellcolor{blue!40} b_{1,0} &
\cellcolor{blue!32} b_{1,1} &
\cellcolor{blue!24} b_{1,2} &
b_{1,3} & \ldots \\
\cellcolor{blue!48} b_{2,0} &
\cellcolor{blue!40} b_{2,1} &
\cellcolor{blue!32} b_{2,2} &
b_{2,3} & \ldots \\
\cellcolor{blue!56} b_{3,0} &
\cellcolor{blue!48} b_{3,1} &
\cellcolor{blue!40} b_{3,2} &
b_{3,3} & \ldots \\
\vdots & \vdots & \vdots & \vdots & \ddots
\end{array}
\]
The upcoming fact was proved by Vasilevski~\cite{Vasilevski2000}.
We obtain it as a corollary from 
Proposition~\ref{prop:basis_in_L2}
and Corollary~\ref{cor:basis_in_true_polyanalytic}.

\begin{cor}\label{cor:decomposition_L2_into_true_polyanalytic}
The space $L^2(\bC,\Gaussian)$ is the orthogonal sum
of the subspaces $\cF_{(m)}$, $m\in\bN$:
\[
L^2(\bC,\Gaussian)=\bigoplus_{m\in\bN}\cF_{(m)}.
\]
\end{cor}

For every $f$ in $\cF_{(n)}$, define $A^\dagger_n f$ by
\[
(A^\dagger_n f)(z)
= \frac{1}{\sqrt{n}} (A^\dagger f)(z)
=\frac{1}{\sqrt{n}} \left(\overline{z} - \frac{\partial}{\partial z}\right) f(z).
\]
Definition \eqref{eq:b_Shigekawa} of the family
$(b_{p,q})_{p,q\in\bNz}$ implies that
\begin{equation}\label{eq:An_b}
A^\dagger_n b_{p,n-1}=b_{p,n}.
\end{equation}
The next picture shows the action of $A^\dagger_2$ on basic elements:
\[
\begin{array}{ccccc}
b_{0,0} & b_{0,1} \mapsto b_{0,2} & b_{0,3} & \ldots \\
b_{1,0} & b_{1,1} \mapsto b_{1,2} & b_{1,3} & \ldots \\
b_{2,0} & b_{2,1} \mapsto b_{2,2} & b_{2,3} & \ldots \\
b_{3,0} & b_{3,1} \mapsto b_{3,2} & b_{3,3} & \ldots \\
\vdots & \vdots\ \ \phantom{\mapsto}\ \ \vdots & \vdots & \ddots
\end{array}
\]

\begin{prop}
$A^\dagger_n$ is an isometric isomorphism
from $\cF_{(n)}$ onto $\cF_{(n+1)}$.
\end{prop}

\begin{proof}
Vasilevski~\cite{Vasilevski2000} proved this fact
by using the Fourier transform.
Here we give another proof.
Write $f$ as in \eqref{eq:f_true_polyanalytic_decomposition}.
It is known \cite[Corollary~1.9]{Balk1991}
that the derivative $\frac{\partial}{\partial\conjz}$
can be applied to the each term of the series.
Therefore
\[
A^\dagger_n f
= \sum_{j=0}^\infty \langle f,b_{j,n-1}\rangle\,A^{\dagger}_n b_{j,n-1}
= \sum_{j=0}^\infty \langle f,b_{j,n-1}\rangle\,b_{j,n},
\]
and $\|A^\dagger_n f\|=\|f\|$.
Also, using the decomposition into series,
we see that $A^\dagger_n$ is surjective.
\end{proof}

Now we are going to prove explicit formulas~\eqref{eq:repr_kernel_true_Fn}
and~\eqref{eq:Kn_formula} for the reproducing kernels
of $\cF_{(n)}$ and $\cF_n$, respectively.
These formulas were published by Balk~\cite[Section~6.3]{Balk1991},
without using the terminology of Laguerre polynomials,
and by Askour, Intissar, and Mouayn \cite{AskourIntissarMouayn1997},
though they defined the space $\cF_{(n)}$ in a different
(but equivalent) way.
Our proof uses the operators $A^\dagger_n$
and thereby continues the work of Vasilevski~\cite{Vasilevski2000}.

\begin{lem}\label{lem: basis of H is square summable}
Let $H$ be a RKHS and
$(e_j)_{j=0}^\infty$ be an ortonormal sequence in $H$.
Then the series $\sum_{j=0}^{\infty}|e_j(z)|^2$ converges.
\end{lem}

\begin{proof}
Denote by $K_{H,z}$ the reproducing kernel of $H$.
From the reproducing property and Bessel's inequality,
\[
\sum_{j=0}^{\infty}|e_j(z)|^2
=\sum_{j=0}^{\infty}|\langle K_{H,z},e_j\rangle|^2
\leq \|K_{H,z}\|^2
=K_{H,z}(z)<\infty.
\qedhere
\]
\end{proof}

\begin{lem}\label{lem:repr_kernel_true_Fn_recursive}
For every $n$ in $\bNz$ and every $z,w$ in $\bC$,
\begin{equation}\label{eq: recurrence reproductive kernel}
K_{(n+1),z}(w) =
\frac{1}{n} \left(z - \frac{\partial}{\partial \overline{z}}\right)
\left(\overline{w} - \frac{\partial}{\partial w}\right)
K_{(n),z}(w).
\end{equation}
\end{lem}

\begin{proof}
It is well known that the reproducing kernel of a RKHS $H$
with an orthonormal basis $(e_j)_{j\in\bNz}$
can be derived from the series
\begin{equation}\label{eq: reproducing kernel as series of base}
K_{H,z}(w)=\sum_{j=0}^\infty\overline{e_j(z)}e_j(w).
\end{equation}
In our case, we use the orthonormal basis
$(b_{p,n})_{p\in\bNz}$ of the space $\cF_{(n+1)}$.
For a fixed $z$ in $\bC$, put $\alpha_p=\overline{b_{p,n}(z)}$.
So,
\[
K_{(n+1),z}
=\sum_{p=0}^\infty \overline{b_{p,n}(z)} b_{p,n}
=\sum_{p=0}^\infty \alpha_p b_{p,n}
=\sum_{p=0}^\infty \alpha_p A_{n-1}^\dagger b_{p,n-1}.
\]
From Lemma~\ref{lem: basis of H is square summable} we know that 
$(\alpha_p)_{p=0}^\infty\in\ell^2$,
thus the series
$\sum_{p=0}^\infty \alpha_p b_{p,n-1}$
converges in $\cF_{(n)}$.
Since $A_{n-1}^\dagger$ is a bounded operator in $\cF_{(n)}$,
we can interchange it with the sum operator.
Therefore
\[
K_{(n+1),z}(z)=\frac{1}{\sqrt{n}}
\left(\conjw-\frac {\partial}{\partial w}\right)
\sum_{p=0}^\infty \overline{b_{p,n}(z)} b_{p,n-1}(w).
\]
Now we fix $w$ in $\bC$,
write $b_{p,n}$ as $A_{n-1}^\dagger b_{p,n-1}$,
and use the fact that $\sum_{p=0}^\infty|b_{p,n-1}(w)|^2<+\infty$.
Following the same ideas as above,
but swapping the roles of $z$ and $w$,
we factorize $\left(\conjw-\frac{\partial}{\partial w}\right)$
from the series:
\[
K_{(n+1),z}(w)
=\frac{1}{n}
\left(z-\frac{\partial}{\partial \overline{z}}\right)
\left(\overline{w}-\frac{\partial}{\partial w}\right)
\sum_{p=0}^\infty \overline{b_{p,n-1}(z)}b_{p,n-1}(w).
\]
The last sum equals $K_{(n),z}(w)$,
which yields~\eqref{eq: recurrence reproductive kernel}.
\end{proof}

\begin{cor}
For every $n$ in $\bNz$ and every $z,w$ in $\bC$,
\begin{equation}\label{eq:repr_ker_creation_operators}
K_{(n),z}(w)=\frac{1}{(n-1)!}
\left(z-\frac{\partial}{\partial\conjz}\right)^{n-1}
\left(\conjw-\frac{\partial}{\partial w}\right)^{n-1}
e^{\conjz w}.
\end{equation}
\end{cor}

\begin{prop}\label{prop:repr_kernel_true_Fn}
The reproducing kernel of $\cF_{(n)}$ is given by
\begin{equation}\label{eq:repr_kernel_true_Fn}
K_{(n),z}(w)=e^{\overline{z}w} L_{n-1}(|w-z|^2).
\end{equation}
\end{prop}

\begin{proof}
Using the definition of creation operators,
formula \eqref{eq:repr_ker_creation_operators} and identity \eqref{eq:derivative_from_Rodriguez_xy} for Laguerre polynomials we have
\begin{align*}
K_{(n),z}(w)
&=
\frac{e^{z\conjz}}{(n-1)!}
\frac{\partial^{n-1}}{\partial\, \conjz^{n-1}}
\left(
e^{-z\conjz}
e^{w\conjw}\frac{\partial^{n-1}}{\partial w^{n-1}}(e^{-w\conjw}
e^{w\conjz})
\right)\\
&=
\frac{e^{z\conjz}}{(n-1)!}
\frac{\partial^{n-1}}{\partial\, \conjz^{n-1}}\left(e^{-\conjz(z-w)}(\conjz-\conjw)^{n-1}\right)\\
&=e^{\conjz w}
\frac{e^{(z-w)(\conjz-w)}}{(n-1)!}
\frac{\partial^{n-1}}{\partial (\conjz-\conjw)^{n-1}}\left(e^{-(\conjz-\conjw)(z-w)}(\conjz-\conjw)^{n-1}\right)\\
&=e^{\conjz w}L_{n-1}(|z-w|^2).\qedhere
\end{align*}
\end{proof}

\begin{cor}
The reproducing kernel of $\cF_n$ is
\begin{equation}\label{eq:Kn_formula}
K_{n,z}(w)=e^{\overline{z}w} L_{n-1}^{(1)}(|w-z|^2).
\end{equation}
\end{cor}

\begin{proof}
Use \eqref{eq:repr_kernel_true_Fn}
and the formula $L_m^{(1)}(x)=\sum_{k=0}^{m-1} L_k(x)$.
\end{proof}

\begin{cor}\label{cor:evaluation_Fn_upper_bound}
For every $f$ in $\cF_n$ and every $z$ in $\bC$,
\begin{equation}\label{eq:evaluation_Fn_upper_bound}
|f(z)|
\le \sqrt{n}\,\econstant^{\frac{|z|^2}{2}}\,\|f\|.
\end{equation}
The equality is achieved when $f=K_{n,z}$.
\end{cor}

\begin{proof}
Indeed,
$\|K_{n,z}\|^2
=K_{n,z}(z)
=\econstant^{|z|^2} L_{n-1}^{(1)}(0)
=n\econstant^{|z|^2}$.
\end{proof}

We finish this section with a couple of simple results
about the Berezin transform and Toeplitz operators in $\cF_n$.
Given a RKHS $H$ over a domain $\Omega$
with a reproducing kernel $(K_z)_{z\in\Omega}$,
the corresponding Berezin transform $\Berezin_H$
acts from $\cB(H)$ to the space $B(\Omega)$ of bounded functions
by the rule
\[
\Berezin_H(S)(z)
=\frac{\langle SK_z,K_z\rangle_H}%
{\langle K_z,K_z\rangle_H}
=\frac{(SK_z)(z)}{K_z(z)}.
\]
Stroethoff proved~\cite{Stroethoff1997}
that $\Berezin_H$ is injective for various RKHS
of analytic functions, in particular, for $H=\cF_1$.
Engli\v{s} noticed \cite[Section~2]{Englis2006} that $\Berezin_H$
is not injective for various RKHS of harmonic functions.
The reasoning of Engli\v{s} can be applied
without any changes to $n$-analytic functions with $n\ge 2$.

\begin{prop}\label{prop:Berezin_not_injective}
Let $n\ge 2$. Then $\Berezin_{\cF_n}$ is not injective.
\end{prop}

\begin{proof}
Let $u$ and $v$ be some linearly independent elements of $\cF_n$
such that $\overline{f},\overline{g}\in\cF_n$.
For example, $u(z)=b_{0,0}(z)=1$
and $v(z)=b_{1,0}(z)=z$.
Following \cite[Section~2]{Englis2006},
consider $S\in\cB(\cF_n)$ given by
\begin{equation}\label{eq:Englis}
Sf \eqdef \langle f,\overline{u}\rangle v
- \langle f,\overline{v}\rangle u.
\end{equation}
With the help of the reproducing property we easily see
that the function $\Berezin_{\cF_n}(S)$ is the zero constant,
although the operator $S$ is not zero.
\end{proof}

Given a measure space $\Omega$ and
a function $g$ in $L^\infty(\Omega)$,
we denote by $M_g$ the multiplication operator
defined on $L^2(\Omega)$ by $M_g f\eqdef gf$.
If $H$ is a closed subspace of $L^2(\Omega)$,
then the \emph{Toeplitz operator} $T_{H,g}$ is defined on $H$ by
\[
T_{H,g}(f)\eqdef P_H(gf)=P_H M_g f.
\]
For $H=\cF_n$ and $H=\cF_{(n)}$,
we write just $T_{n,g}$ and $T_{(n),g}$, respectively.

\begin{prop}\label{prop:zero_Toeplitz_operator}
Let $g\in L^\infty(\bC)$ and $T_{n,g}=0$. Then $g=0$ a.e.
\end{prop}

\begin{proof}
For $n=1$, this result was proven in \cite[Theorem~4]{BergerCoburn1986}.
Let us recall that proof which also works for $n\ge 2$.
The condition $T_{n,g}=0$ implies that for all $j,k$ in $\bNz$
\[
\langle g,m_{j,k}\rangle
=\int_\bC g(z)\,\conjz^j z^k\,\dif\Gaussian(z)
=\langle g m_{k,0},m_{j,0}\rangle
=\langle T_{n,g}m_{k,0},m_{j,0}\rangle
=0.
\]
Since $\{m_{j,k}\colon\ j,k\in\bNz\}$ is a dense subset of $L^2(\bC,\Gaussian)$, $g=0$ a.e.
\end{proof}

\section{Unitary representations defined by changes of variables}
\label{sec:unitary_representations}

This section states some simple general facts
about unitary group representations in RKHS,
defined by changes of variables.
Suppose that $(\Omega,\nu)$ is a measure space,
$H$ is a RKHS over $\Omega$,
with the inner product inherited from $L^2(\Omega)$,
$(K_z)_{z\in\Omega}$ is the reproducing kernel of $H$,
and $P_H\in\cB(L^2(\Omega))$ is the orthogonal projection
whose image is $H$:
\[
(P_H f)(z)=\langle f,K_z\rangle_{L^2(\Omega)}.
\]
Furthermore, let $G$ be a locally compact group,
and $\alpha$ be a group action in $\Omega$.
So, for every $\tau$ in $G$ we have a ``change of variables''
$\alpha(\tau)\colon\Omega\to\Omega$,
which satisfies
$\alpha(\tau_1\tau_2)=\alpha(\tau_1)\circ\alpha(\tau_2)$.
Suppose that the function $\rho$,
defined by the following rule,
is a strongly continuous unitary representation
of the group $G$ in the space $L^2(\Omega)$:
\[
\rho(\tau)f\eqdef f\circ\alpha(\tau^{-1})\qquad
(f\in L^2(\Omega),\ \tau\in G).
\]
In other words, we suppose that
$\rho(\tau) f\in L^2(\Omega)$,
$\|\rho(\tau) f\|_{L^2(\Omega)}=\|f\|_{L^2(\Omega)}$,
and $\rho(\tau) f$ depends continuously on $\tau$.

\begin{prop}\label{prop:criterion_invariant_subspace}
The following conditions are equivalent.
\begin{itemize}
\item[(a)] $\rho(\tau)(H)\subseteq H$ for every $\tau$ in $G$.
\item[(b)] $\rho(\tau)P_H = P_H\rho(\tau)$ for every $\tau$ in $G$.
\item[(c)] The reproducing kernel is
invariant under simultaneous changes of variables
in both arguments:
\[
K_{\alpha(\tau)(z)}(\alpha(\tau)(w))
=K_z(w)\qquad(\tau\in G,\ z,w\in\Omega).
\]
\item[(d)] $\rho(\tau) K_z = K_{\alpha(\tau)(z)}$
for every $z$ in $\Omega$ and every $\tau$ in $G$.
\end{itemize}
\end{prop}

\begin{proof}
Obviously, (a) is equivalent to (b).
Suppose (a) and prove (c):
\begin{align*}
K_{\alpha(\tau)(z)}(\alpha(\tau)(w))
&=(\rho(\tau^{-1}) K_{\alpha(\tau)(z)})(w)
=\langle \rho(\tau^{-1}) K_{\alpha(\tau)(z)},K_w\rangle_{L^2(\Omega)}
\\
&=\langle K_{\alpha(\tau)(z)},\rho(\tau) K_w\rangle_{L^2(\Omega)}
=\overline{(\rho(\tau) K_w)(\alpha(\tau)(z))}
=\overline{K_w(z)}=K_z(w).
\end{align*}
Suppose (c) and prove (d):
\[
(\rho(\tau) K_z)(w)
=K_z(\alpha(\tau^{-1})(w))
=K_{\alpha(\tau)(z)}(\alpha(\tau)(\alpha(\tau^{-1})(w)))
=K_{\alpha(\tau)(z)}(w).
\]
Suppose (d) and prove (a).
Let $f\in H$. Then
\begin{align*}
(\rho(\tau) f)(z)
&=f(\alpha(\tau^{-1})(z))
=\langle f,K_{\alpha(\tau^{-1})(z)}\rangle_{L^2(\Omega)}
\\
&=\langle \rho(\tau) f, \rho(\tau) K_{\alpha(\tau^{-1})(z)}\rangle_{L^2(\Omega)}
=\langle \rho(\tau) f, K_z\rangle_{L^2(\Omega)}.
\qedhere
\end{align*}
\end{proof}

Suppose that the conditions (a)--(d) of Proposition~\ref{prop:criterion_invariant_subspace} are fulfilled.
For every $\tau$ in $G$ we denote by $\rho_H(\tau)$
the compression of the operator $\rho(\tau)$
to the invariant subspace $H$.
Then $\rho_H$ is a unitary representation of $G$ in $H$. 
Let us relate this unitary representation
with the Berezin transform of operators.

\begin{prop}\label{prop:Berezin_and_group_action}
Let $S\in\cB(H)$ and $\tau\in G$.
Then
\begin{equation}\label{eq:Berezin_and_group_action}
\Berezin_H(\rho_H(\tau^{-1}) S \rho_H(\tau))(z)
=\Berezin_H(S)(\alpha(\tau)(z))\qquad(z\in\Omega).
\end{equation}
\end{prop}

\begin{proof}
\begin{align*}
\Berezin_H(\rho_H(\tau^{-1}) S \rho_H(\tau))(z)
&=\frac{(\rho_H(\tau^{-1}) S \rho_H(\tau) K_z)(z)}{K_z(z)}
\\
&=\frac{(S K_{\alpha(\tau)(z)}(\alpha(\tau)(z))}%
{K_{\alpha(\tau)(z)}(\alpha(\tau)(z))}
=\Berezin_H(S)(\alpha(\tau)(z)).
\qedhere
\end{align*}
\end{proof}

\begin{cor}\label{cor:Berezin_of_invariant_operator}
Let $S\in\cB(H)$ such that $S\rho(\tau)=\rho(\tau)S$
for every $\tau$ in $G$.
Then the function $\Berezin_H(S)$ is invariant under $\alpha$,
i.e. $\Berezin_H(S)\circ\alpha(\tau)=\Berezin_H(S)$
for every $\tau$ in $G$.
\end{cor}

If $\Berezin_H$ is injective,
then the inverse of the Corollary~\ref{cor:Berezin_of_invariant_operator} is also true.

The rest of this section does not assume
that $H$ has a reproducing kernel;
it can be just a closed subspace of $L^2(\Omega)$.

We are going to state some elementary results
about the interaction of $\rho_H$ with Toeplitz operators.
These results are well known for many particular cases;
see \cite[Lemma~3.2 and Corollary~3.3]{DawsonOlafssonQuiroga2015}
for the case when $H$ is a Bergman space of analytic functions.

\begin{lem}\label{lem:mul_operator_and_group_representation}
Let $g\in L^\infty(\Omega)$ and $\tau\in G$.
Then
\[
M_g \rho(\tau) = \rho(\tau) M_{g\circ\alpha(\tau)}.
\]
\end{lem}

\begin{proof}
Put $u\eqdef g\circ\alpha(\tau)$.
Given $f$ in $L^2(\Omega)$,
\[
M_g \rho(\tau) f
= (u\circ\alpha(\tau^{-1}))\,(f\circ\alpha(\tau^{-1}))
= (uf)\circ\alpha(\tau^{-1})
= \rho(\tau) M_u f.
\qedhere
\]
\end{proof}

\begin{prop}\label{prop:Toeplitz_operator_and_group_representation}
Let $g\in L^\infty(\Omega)$ and $\tau\in G$.
Then
\begin{equation}\label{eq:Toeplitz_operator_and_group_representation}
T_g \rho_H(\tau) = \rho_H(\tau) T_{g\circ\alpha(\tau)}.
\end{equation}
\end{prop}

\begin{proof}
Use Lemma~\ref{lem:mul_operator_and_group_representation}
and the assumption $P_H \rho(\tau)=\rho(\tau) P_H$:
\[
T_g \rho_H(\tau) f
= P_H M_g \rho(\tau) f
= P_H \rho(\tau) M_{g\circ\al(\tau)} f
= \rho(\tau) P_H M_{g\circ\al(\tau)} f
= \rho_H(\tau) T_{g\circ\al(\tau)}f.
\qedhere
\]
\end{proof}

\begin{cor}\label{cor:Toeplitz_with_invariant_symbol_is_invariant_under_group_representation}
Let $g\in L^\infty(\Omega)$ such that
$g\circ\alpha(\tau)=g$ for every $\tau$ in $G$.
Then $T_g$ commutes with $\rho_H(\tau)$ for every $\tau$ in $G$.
\end{cor}

\begin{cor}\label{cor:Toeplitz_invariant_under_group_representation}
Suppose that the mapping $L^\infty(\Omega)\to\cB(H)$
defined by $a\mapsto T_a$ is injective.
Let $g\in L^\infty(X)$ such that
$T_g$ commutes with $\rho_H(\tau)$ for every $\tau$ in $G$.
Then for every $\tau$ in $G$
the functions $g\circ\alpha(\tau)$ and $g$ coincide a.e.
\end{cor}

\section{Von Neumann algebras of radial operators}
\label{sec:radial}

The methods of this section are similar to ideas from
\cite{Zorboska2003,GrudskyMaximenkoVasilevski2013,Quiroga2016}.
We start with two simple general schemes,
stated in the context of von Neumann algebras,
and then apply them to radial operators in $L^2(\Omega,\gamma)$,
in $\cF_n$, and in $\cF_{(n)}$.
Proposition~\ref{prop:diagonalizing_family_of_subspaces}
uses the concept of the (bounded) direct sum of von Neumann algebras
\cite[Definition~1.1.5]{Sakai1971}.

\begin{defn}\label{def:diagonalizing_family_of_subspaces}
Let $H$ be a Hilbert space,
$\cU$ be a self-adjoint subset of $\cB(H)$,
and $(W_j)_{j\in J}$ be a finite or countable family
of nonzero closed subspaces of $H$ such that
$H=\bigoplus_{j\in J}W_j$.
We say that this family \emph{diagonalizes} $\cU$ if
the following two conditions are satisfied.
\begin{enumerate}
\item For each $j$ in $J$ and each $U$ in $\cU$,
there exists $\la_{U,j}$ in $\bC$ such that
$W_j\subseteq\ker(\la_{U,j}I-U)$,
i.e. $U(v)=\la_{U,j}v$ for every $v$ in $W_j$.
\item For every $j$, $k$ in $J$ with $j\ne k$,
there exists $U$ in $\cU$ such that $\la_{U,j}\ne\la_{U,k}$.
\end{enumerate}
\end{defn}

\begin{prop}\label{prop:diagonalizing_family_of_subspaces}
Let $H$, $\cU$, and $(W_j)_{j\in J}$ be like in
Definition~\ref{def:diagonalizing_family_of_subspaces}.
Denote by $\cA$ the commutant of $\cU$.
Then
\begin{equation}\label{eq:conmutant_description}
\cA=\{S\in\cB(H)\colon\quad\forall j\in J\quad S(W_j)\subseteq W_j\},
\end{equation}
and $\cA$ is isometrically isomorphic to
$\bigoplus_{j\in J}\cB(W_j)$.
\end{prop}

\begin{proof}
1. Since $\cU$ is a self-adjoint subset of $\cB(H)$,
its commutant $\cA$ is a von Neumann algebra
\cite[Proposition 18.1]{Zhu1993}.

2. Notice that if $U\in\cU$ and $j\in J$, then
$\la_{U^\ast,j}=\overline{\la_{U,j}}$.
Indeed, for every $v$ in $W_j\setminus\{0\}$
\[
\la_{U,j}\|v\|_H^2
=\langle\la_{U,j}v,v\rangle_H
=\langle Uv,v\rangle_H
=\langle v,U^\ast v\rangle_H
=\langle v,\la_{U^\ast,j}v\rangle_H
=\overline{\la_{U^\ast,j}}\,\|v\|_H^2.
\]
3. Let $S\in\cA$, $j\in J$, $f\in W_j$.
We are going to prove that $Sf\in W_j$.
If $k\in J\setminus\{j\}$ and $g\in W_k$,
then there exists $U$ in $\cU$ with $\la_{U,j}\ne\la_{U,k}$, and
\[
\la_{U,j}\langle Sf,g\rangle_H
=\langle SUf,g\rangle_H
=\langle USf,g\rangle_H
=\langle Sf,U^\ast g\rangle_H
=\la_{U,k}\langle Sf,g\rangle_H.
\]
which implies that $\langle Sf,g\rangle_H=0$.
Since $H=\bigoplus_{k\in H}W_j$,
the vector $Sf$ expands into the series of the form
$Sf=\sum_{q\in J}h_q$ with $h_k\in W_k$.
For every $k$ in $J\setminus\{j\}$,
\[
0=\langle Sf,h_k\rangle_H
=\langle h_k,h_k\rangle_H
+\sum_{q\in J\setminus\{k\}}\langle h_q,h_k\rangle_H
=\|h_k\|_H^2.
\]
Thus, $Sf=h_j\in W_j$.

4. Now suppose that $S\in\cB(H)$ and
$S(W_j)\subseteq W_j$ for every $j\in J$.
Then for every $U$ in $\cU$, $j$ in $J$, and $g$ in $W_j$,
\[
USg=U(Sg)=\la_{U,j}Sg=S(\la_{U,j}g)=SUg.
\]
In general, if $f$ in $H$, then $f=\sum_{j\in J}g_j$
with some $g_j$ in $W_j$, and
\[
USf=\sum_{j\in J}USg_j=\sum_{j\in J}SUg_j=SUf.
\]
5. Using \eqref{eq:conmutant_description}
we are going to prove that $\cA$
is isometrically isomorphic to $\bigoplus_{j\in J}\cB(W_j)$.
Given $S$ in $\cA$, for every $j$ in $J$
we denote by $A_j$ the compression of $S$
onto the invariant subspace $W_j$.
Then the family $(A_j)_{j\in J}$
belongs to $\bigoplus_{d\in J}\cB(W_d)$, and
$\|S\|=\sup_{j\in J}\|A_j\|$.

Conversely, given a bounded sequence $(A_j)_{j\in J}$
with $A_j$ in $\cB(W_j)$, we put
\[
S\left(\sum_{j\in J}g_j\right)
=\sum_{j\in J}A_j g_j\qquad(g_j\in W_j).
\]
Then $S(W_j)\subseteq W_j$ for every $j$ in $J$, thus $S\in\cA$.
Thereby we have constructed isometrical isomorphisms between $\cA$
and $\bigoplus_{j\in J}\cB(W_j)$.
\end{proof}

Proposition~\ref{prop:diagonalizing_family_of_subspaces}
implies that the von Neumann algebra generated by $\cU$
consists of all operators
that act as scalar operators on each $W_j$,
and can be naturally identified with
$\bigoplus_{j\in J}\bC I_{W_j}$.

\begin{prop}\label{prop:diagonalizing_family_in_subspace}
Let $H$, $\cU$, and $(W_j)_{j\in J}$ be like in
Definition~\ref{def:diagonalizing_family_of_subspaces},
and $H_1$ be a closed subspace of $H$ invariant under $\cU$.
For every $U$ in $\cU$, denote by $U_1$ the compression
of $U$ onto the invariant subspace $H_1$, and put
\[
\cU_1\eqdef\{U_1\colon\ U\in\cU\},\qquad
J_1\eqdef\{j\in J\colon\ W_j\cap H_1\ne\{0\}\}.
\]
Then
\begin{equation}\label{eq:H1_orthogonal_decomposition}
H_1=\bigoplus_{j\in J_1}(W_j\cap H_1),
\end{equation}
and the family $(W_j\cap H_1)_{j\in J}$ diagonalizes $\cU_1$.
\end{prop}

\begin{proof}
Denote by $P_1$ the orthogonal projection that acts in $H$
and has image $H_1$.
The condition that $H_1$ is invariant under $\cU$
means that $P_1\in\cA$.
By~\eqref{eq:conmutant_description},
for every $j$ in $J$ the subspace $P_1(W_j)$
is contained in $W_j$ and therefore coincides with $W_j\cap H_1$.
This easily implies~\eqref{eq:H1_orthogonal_decomposition}.

If $U\in\cU$ and $j\in J$,
then $W_j\cap H_1\subseteq\ker(\la_{U,j}I_{H_1}-U_1)$.
So, the eigenvalues $\la_{U_1,j}$ coincide
with $\la_{U,j}$ for every $j$ in $J_1$.

If $j,k\in J_1$ and $j\ne k$,
then there exists $U$ in $\cU$ such that $\la_{U,j}\ne\la_{U,k}$,
which means that $\la_{U_1,j}\ne\la_{U_1,k}$.
\end{proof}

\subsection*{\texorpdfstring{Radial operators in $\boldsymbol{L^2(\bC,\Gaussian)}$}{Radial operators in L2(C,Gaussian)}}

For each $\tau$ in $\bT$, denote by $R_\tau$
the rotation operator acting in $L^2(\bC,\gamma)$:
\begin{equation}\label{eq:def_rotation}
(R_\tau f)(z) = f(\tau^{-1}z).
\end{equation}
The family $(R_\tau)_{\tau\in\bT}$
is a unitary representation of the group $\bT$ in $L^2(\bC,\gamma)$.
Notice that we are in the situation of Section~\ref{sec:unitary_representations},
with $\Omega=\bC$, $\nu=\Gaussian$, $G=\bT$,
$\alpha(\tau)(z)=\tau z$, $\rho(\tau)=R_\tau$.

Denote by $\cR$ the set of all radial operators acting in $L^2(\bC,\gamma)$:
\[
\cR = \{S\in\cB(L^2(\bC,\gamma))\colon\quad
\forall\tau\in\bT\quad R_\tau S=S R_\tau\}.
\]
Since the set $\{R_\tau\colon\ \tau\in\bT\}$
is an autoadjoint subset of $\cB(L^2(\bC,\Gaussian))$,
its commutant $\cR$ is a von Neumann algebra.

\begin{lem}\label{lem:diagonal_subspaces_diagonalize_rotation_operators}
The family $(\cD_d)_{d\in\bZ}$ diagonalizes
the collection $\{R_\tau\colon\ \tau\in\bT\}$
in the sense of
Definition~\ref{def:diagonalizing_family_of_subspaces}.
\end{lem}

\begin{proof}
If $\tau\in\bT$ and $d\in\bZ$, then
\begin{equation}\label{eq:diagonal_subspaces_as_eigensubspaces}
\cD_d\subseteq\ker(\tau^{-d} I-R_\tau).
\end{equation}
Indeed, for every $p,q\in\bZ$ with $p-q=d$
the basic function $b_{p,q}$ is an eigenfunction of $R_\tau$
associated to the eigenvalue $\tau^{-d}$:
\begin{equation}\label{eq:rotation_on_the_basics}
R_\tau b_{p,q} = \tau^{q-p} b_{p,q}=\tau^{-d}b_{p,q},
\end{equation}
and by Corollary~\ref{cor:basis_of_D}
the functions $b_{p,q}$ with $p-q=d$
form an orthonormal basis of $\cD_d$.
Another way to prove~\eqref{eq:diagonal_subspaces_as_eigensubspaces}
is to use~Corollary~\ref{cor:polar_description_of_D}.

If $d_1,d_2\in\bZ$ and $d_1\ne d_2$,
then $\tau^{-d_1}\ne\tau^{-d_2}$ for many values of $\tau$,
for example, for $\tau=\econstant^{\frac{\imagunit\pi}{d_1-d_2}}$
or for $\tau=\econstant^{\imagunit\tht}$ with any irrational $\tht$.
\end{proof}

\begin{prop}\label{prop:radial_operators_on_L2}
The von Neumann algebra $\cR$ consists of all operators
that act invariantly on $\cD_d$ for every $d$ in $\bZ$,
and is isometrically isomorphic to $\bigoplus_{d\in\bZ}\cB(\cD_d)$.
\end{prop}

\begin{proof}
This is a consequence of
Proposition~\ref{prop:diagonalizing_family_of_subspaces}
and Lemma~\ref{lem:diagonal_subspaces_diagonalize_rotation_operators}.
\end{proof}

Now we will describe all radial operators of finite rank.

\begin{remark}\label{rem:finite_rank_operators}
It is well known that every linear operator of a finite rank $m$,
acting in a Hilbert space $H$, can be written in the form
\begin{equation}\label{eq:finite_rank}
Sf=\sum_{k=1}^m \xi_k \langle f,u_k\rangle_H v_k,
\end{equation}
where $\xi_1,\ldots,\xi_m\in\bC\setminus\{0\}$,
$u_1,\ldots,u_m$ and $v_1,\ldots,v_m$
are some orthonormal lists of vectors in $H$.
\end{remark}

\begin{cor}\label{cor:radial_operators_of_finite_rank_in_L2}
Let $m\in\bN$ and $S\in\cB(L^2(\bC,\Gaussian))$
such that the rank of $S$ is $m$.
Then $S$ is radial if and only if
there exist $d_1,\ldots,d_m$ in $\bZ$
such that $S$ has the form~\eqref{eq:finite_rank},
where $u_j$, $v_j$, $\xi_j$ are like in Remark~\ref{rem:finite_rank_operators},
and additionally $u_j,v_j\in\cD_{d_j}$ for every
$j$ in $\{1,\ldots,m\}$.
\end{cor}

\begin{proof}
This is a simple consequence of Proposition~\ref{prop:radial_operators_on_L2}.
Suppose that $S$ is radial.
For every $d$ in $\bZ$ let $A_d$ be the compression of $S$ to $\cD_d$.
There is only a finite set of $d$ such that $A_d\ne0$.
Apply Remark~\ref{rem:finite_rank_operators}
to each of the nonzero operators $A_d$
and join the obtained decompositions.
\end{proof}

Following Zorboska~\cite{Zorboska2003},
we will describe radial operators in term of the ``radialization''
$\Radialization\colon\cB(L^2(\bC,\Gaussian))
\to\cB(L^2(\bC,\Gaussian))$ defined by
\[
\Radialization(S)\eqdef
\int_\bT R_\tau S R_{\tau^{-1}}\,\dif\mu_\bT(\tau),
\]
where $\mu_\bT$ is the normalized Haar measure on $\bT$.
The integral is understood in the weak sense,
i.e. the operator $\Radialization(S)$
is actually defined by the equality of the corresponding
sesquilinear forms:
\[
\langle\Radialization(S)f,g\rangle
=\int_\bT \langle R_\tau S R_{\tau^{-1}}f,g\rangle\,\dif\mu_\bT(\tau).
\]
Making an appropriate change of variables in the integral
and using the invariance of the measure $\mu_\bT$,
we see that $\Radialization(S)\in\cR$.
This immediately implies the following criterion
of radial operators in terms of the radialization.

\begin{prop}\label{prop:criterion_radial_in_terms_of_radialization}
Let $S\in\cB(L^2(\bC,\Gaussian))$.
Then $S\in\cR$ if and only if $\Radialization(S)=S$.
\end{prop}

\subsection*{\texorpdfstring{Radial operators in $\boldsymbol{\cF_n}$}{Radial operators in Fn}}

Let $n\in\bN$.
Obviously, the reproducing kernel of $\cF_n$,
given by \eqref{eq:Kn_formula},
is invariant under simultaneous rotations in both arguments:
\begin{equation}\label{eq:Kn_and_rotations}
K_{n,\tau z}(\tau w)=K_{n,z}(w)\qquad(z,w\in\bC,\ \tau\in\bT).
\end{equation}
Therefore, by Proposition~\ref{prop:criterion_invariant_subspace},
$\cF_n$ is invariant under rotations, and $P_n\in\cR$.
For every $\tau$ in $\bT$,
we denote by $R_{n,\tau}$ the compression of $R_\tau$
onto the space $\cF_n$.
In other words, the operator $R_{n,\tau}$ acts in $\cF_n$
and is defined by \eqref{eq:def_rotation}.
The family $(R_{n,\tau})_{\tau\in\bT}$
is a unitary representation of $\bT$ in $\cF_n$.
Let $\cR_n$ be the von Neumann algebra
of all bounded linear radial operators acting in $\cF_n$.

Denote by $\fM_n$ the following direct sum of matrix algebras:
\[
\fM_n \eqdef \bigoplus_{d=-n+1}^\infty \Mat_{\min\{n,n+d\}}
= \left(\bigoplus_{d=-n+1}^{-1} \Mat_{n+d}\right)
\oplus
\left(\bigoplus_{d=0}^{\infty} \Mat_{n}\right).
\]
The elements of $\fM_n$ are matrix sequences of the form
$A=(A_d)_{d=-n+1}^\infty$,
where $A_d\in\Mat_{n+d}$ if $d<0$, $A_d\in\Mat_n$ if $d\ge0$, and
\[
\sup_{d\ge-n+1}\|A_d\|<+\infty.
\]
Now we are ready to prove Theorem~\ref{thm:radial_polyanalytic_Fock}.

\begin{proof}[Proof of Theorem~\ref{thm:radial_polyanalytic_Fock}.]
By~Propositions~\ref{prop:diagonalizing_family_of_subspaces},
\ref{prop:diagonalizing_family_in_subspace}
and formula~\eqref{eq:truncated_diagonal_as_intersection},
$\cR_n$ is isometrically isomorphic to the direct sum
of $\cB(\cD_{d,\min\{n,n+d\}})$, with $d\ge-n+1$.
Using the orthonormal basis $(b_{d+k,k})_{k=\max\{0,-d\}}^{n-1}$
of the space $\cD_{d,\min\{n,n+d\}}$,
we represent linear operators on this space as matrices.
Define $\Phi_n\colon\cR_n\to\fM_n$ by
\begin{equation}\label{eq:Phi}
\Phi_n(S)=\left(\left[\left\langle Sb_{d+k,k},b_{d+j,j}\right\rangle\right]_{j,k=\max\{0,-d\}}^{n-1}\right)_{d=-n+1}^{\infty}.
\end{equation}
Then $\Phi_n$ is an isometrical isomorphism.
\end{proof}

Similarly to Corollary~\ref{cor:radial_operators_of_finite_rank_in_L2}, there is a simple description of radial operators
of finite rank acting in $\cF_n$.
Of course, now $d_1,\ldots,d_m\ge-n+1$.

By Corollary~\ref{cor:Berezin_of_invariant_operator},
if $S\in\cR_n$, then $\Berezin_{\cF_n}(S)$ is a radial function.
For $n=1$, the Berezin transform $\Berezin_{\cF_1}$
is injective.
So, if $S\in\cB(\cF_1)$ and the function $\Berezin_{\cF_1}(S)$ is radial, then $S\in\cR_1$.
For $n\ge 2$, there are nonradial operators $S$
with radial Berezin transforms.

\begin{example}\label{example:nonradial_operator_with_radial_Berezin_transform}
Let $n\ge 2$.
Define $u$, $v$, and $S$ like in the proof of Proposition~\ref{prop:Berezin_not_injective}.
Then $\Berezin(S)$ is the zero constant.
In particular, $\Berezin(S)$ is a radial function.
On the other hand, $Sb_{0,0}=b_{1,0}$,
the subspace $\cD_0$ is not invariant under $S$,
and thus $S$ is not radial.
\end{example}

\subsection*{\texorpdfstring{Radial operators in $\boldsymbol{\cF_{(n)}}$}{Radial operators in F(n)}}

Let $n\in\bN$.
By~Proposition~\ref{prop:criterion_invariant_subspace}
and formula~\eqref{eq:repr_kernel_true_Fn},
the subspace $\cF_{(n)}$ is invariant under the rotations
$R_\tau$ for all $\tau$ in $\bT$.
Denote the corresponding compression of $R_\tau$ by $R_{(n),\tau}$.
Let $\cR_{(n)}$ be the von Neumann algebra of all radial operators in $\cF_{(n)}$.

\begin{proof}[Proof of Theorem~\ref{thm:radial_true_polyanalytic_Fock}]
Corollaries~\ref{cor:basis_of_D} and \ref{cor:basis_in_true_polyanalytic} give
\begin{equation}\label{eq:intersection_diagonal_with_column}
\cD_d\cap\cF_{(n)}
=\begin{cases}
\bC b_{d+n-1,n-1}, & d\ge-n+1,\\
\{0\}, & d<-n+1.
\end{cases}
\end{equation}
By~Propositions~\ref{prop:diagonalizing_family_of_subspaces},
\ref{prop:diagonalizing_family_in_subspace}
and formula~\eqref{eq:intersection_diagonal_with_column},
$\cR_{(n)}$ consists of the operators
that act invariantly on $\bC b_{d+n-1,n-1}$, $d\ge-n+1$,
i.e. are diagonal with respect to the basis
$(b_{p,n-1})_{p=0}^\infty$.
Therefore the function
$\Phi_{(n)}\colon\cR_{(n)}\to\ell^\infty(\bNz)$,
defined by
\begin{equation}\label{eq:Phi_true}
\Phi_{(n)}(S)
=\bigl(\langle S b_{p,n-1},b_{p,n-1}\rangle\bigr)_{p=0}^\infty,
\end{equation}
is an isometric isomorphism.
\end{proof}

Similarly to Corollary~\ref{cor:radial_operators_of_finite_rank_in_L2}, there is a simple description of radial operators
of finite rank acting in $\cF_{(n)}$.

\section{Radial Toeplitz operators in polyanalytic spaces}
\label{sec:Toeplitz}

A measurable function $g\colon\bC\to\bC$ is called \emph{radial}
if for every $\tau$ in $\bT$ the equality $g(\tau z)=g(z)$
is true for a.e. $z$ in $\bC$.
If $g\in L^2(\bC,\Gaussian)$, then this condition means that
$R_\tau g=g$ for every $\tau$ in $\bT$.

Given a function $a$ in $L^\infty([0,+\infty))$,
let $\widetilde{a}$ be its extension
defined on $\bC$ as
\[
\widetilde{a}(z) \eqdef a(|z|)\qquad(z\in\bC).
\]
It is easy to see that a function $g$ in $L^\infty(\bC)$ is radial
if and only if there exists $a$ in $L^\infty([0,+\infty))$
such that $g=\widetilde{a}$.

By~Lemma~\ref{lem:mul_operator_and_group_representation},
the multiplication operator $M_{\widetilde{a}}$,
acting in $L^2(\bC,\gamma)$, is radial.
Let us compute the matrix of this operator with respect to the basis $(b_{p,q})_{p,q\in\bNz}$.
Put
\[
\beta_{a,d,j,k}\eqdef
\langle \widetilde{a} b_{j+d,j},b_{k+d,k}\rangle
\qquad(d\in\bZ,\ j,k,j+d,k+d\in\bNz).
\]
Passing to the polar coordinates
and using~\eqref{eq:b_via_Laguerre_function} we get
\begin{equation}\label{eq:beta_as_integral}
\beta_{a,d,j,k}
=\int_0^{+\infty} a(\sqrt{t})\,
\ell_{\min\{j,j+d\}}^{(|d|)}(t)
\ell_{\min\{k,k+d\}}^{(|d|)}(t)\,\dif{}t.
\end{equation}

\begin{prop}\label{prop:radial_multiplication}
Let $a\in L^\infty([0,+\infty))$.
Then $M_{\widetilde{a}}\in\cR$, and
\[
\langle M_{\widetilde{a}}b_{p,q},b_{j,k}\rangle
=\langle \widetilde{a} b_{p,q},b_{j,k}\rangle
=\delta_{p-q,j-k} \beta_{a,p-q,q,k}.
\]
\end{prop}

\begin{proof}
Use the fact that $M_{\widetilde{a}}$ is radial
and the orthogonality of the ``diagonal subspaces''.
Then apply the definition of $\beta_{a,d,j,k}$.
\end{proof}

\begin{prop}\label{prop:radial_Toeplitz_operator}
Let $g\in L^\infty(\bC)$.
Then the opeator $T_{n,g}$ is radial if and only if
the function $g$ is radial.
\end{prop}

\begin{proof}
Apply Proposition~\ref{prop:zero_Toeplitz_operator}
and Corollaries~\ref{cor:Toeplitz_with_invariant_symbol_is_invariant_under_group_representation}, \ref{cor:Toeplitz_invariant_under_group_representation}.
\end{proof}

\begin{prop}\label{prop:radial_Toeplitz_eigenvalues_true_poly_Fock}
Let $a\in L^\infty([0,+\infty))$.
Then $T_{(n),\widetilde{a}}\in\cR_{(n)}$,
the operator $T_{(n),\widetilde{a}}$ is diagonal
with respect to the orthonormal basis
$(b_{p,n-1})_{p=0}^\infty$,
and the sequence $\lambda_{a,n}$ of the corresponding eigenvalues
can be computed by
\begin{equation}\label{eq:radial_Toeplitz_eigenvalues_true_poly_Fock}
\lambda_{a,n}(p)
=\beta_{a,p-n+1,n-1,n-1}
=\int_0^{+\infty} a(\sqrt{t})\,
\bigl(\ell_{\min\{p,n-1\}}^{(|p-n+1|)}(t)\bigr)^2\,\dif{}t
\qquad(p\in\bNz).
\end{equation}
\end{prop}

\begin{proof}
From Corollary~\ref{cor:Toeplitz_with_invariant_symbol_is_invariant_under_group_representation} we get
$T_{(n),\widetilde{a}}\in\cR_{(n)}$.
Due to Proposition~\ref{prop:radial_multiplication}
and Theorem~\ref{thm:radial_true_polyanalytic_Fock},
\[
\la_{a,n}(p)
=(\Phi_{(n)}(T_{(n),\widetilde{a}}))_p
=\langle T_{(n),\widetilde{a}}b_{p,n-1},b_{p,n-1}\rangle
=\beta_{a,p-n+1,n-1,n-1}.
\qedhere
\]
\end{proof}

Given a class $G\subseteq L^\infty(\bC)$ of generating symbols,
we denote by $\cT_{(n)}(G)$ the C*-subalgebra of $\cB(\cF_{(n)})$
generated by the set $\{T_{(n),g}\colon\ g\in G\}$.
Let $\RB$ be the space of all radial bounded functions on $\bC$,
and $\RBC$ be the space of all radial bounded functions on $\bC$
having a finite limit at infinity.

We are going to describe the algebra $\cT_{(n)}(\RBC)$.

\begin{lem}\label{lem:Laguerre_tend_to_zero}
Let $m\in\bNz$ and $x>0$. Then
\[
\lim_{d\to\infty}\sup_{0\le t\le x}|\ell_m^{(d)}(t)|=0.
\]
\end{lem}

\begin{proof}
For each $t<x$, we write $\ell_{m}^{(d)}(t)$ explicitly
by~\eqref{eq:laguerre_function} and \eqref{eq:generalized_Laguerre_explicit},
then apply simple upper bounds:
\begin{align*}
|\ell_{m}^{(d)}(t)|
&=\sqrt{\frac{m!}{(m+d)!}}\,|t^{\frac{d}{2}}\econstant^{-\frac{t}{2}}L_{m}^{(d)}(t)|
\le\sqrt{\frac{m!}{(m+d)!}}\econstant^{-t/2}
\sum_{j=0}^{m}\binom{m+d}{m-j}\frac{t^{j+\frac{d}{2}}}{j!}\\
&\leq
\sqrt{\frac{m!}{(m+d)!}}\sum_{j=0}^{m}\frac{(m+d)!}{(d+j)!}t^{j+\frac{d}{2}}
\leq
(m+1)\sqrt{m!}\,\frac{(m+d)^{m}\,(1+t)^{m+\frac{d}{2}}}{\sqrt{(m+d)!}}.
\end{align*}
Then,
\[
\sup_{0\le t\le x}|\ell_m^{(d)}(t)|
\le \frac{\sqrt{m!}\,(m+1)\,(m+d)^m\,(1+x)^{m+\frac{d}{2}}}{\sqrt{(m+d)!}},
\]
and the last expression tends to $0$ as $d$ tends to $\infty$.
\end{proof}

The following lemma and proposition are similar to
\cite[Lemma~7.2.3 and Theorem~7.2.4]{Vasilevski2008book}.

\begin{lem}\label{lem:limit_beta}
Let $a\in L^\infty([0,+\infty))$, $v\in\bC$,
and $\lim\limits_{r\to+\infty}a(r)=v$.
Then
\begin{equation}\label{eq:limit_beta}
\lim_{d\to+\infty}\beta_{a,d,j,k}=\delta_{j,k} v\qquad
(j,k\in\bNz).
\end{equation}
In particular,
\begin{equation}\label{eq:limit_beta_diagonal}
\lim_{p\to\infty}\lambda_{a,n}(p)=v
\qquad(n\in\bN).
\end{equation}
\end{lem}

\begin{proof}
1. First, suppose that $v=0$ and $j=k$.
For every $x>0$ and $d\ge 0$,
\begin{align*}
|\beta_{a,d,j,j}|
&\le \int_0^x |a(\sqrt{t})|\,
\bigl(\ell_j^{(d)}(t)\bigr)^2\,\dif{}t
+\int_x^{+\infty} |a(\sqrt{t})|\,
\bigl(\ell_j^{(d)}(t)\bigr)^2\,\dif{}t
\\
&\le x \|a\|_\infty \left(\,\sup_{0\le t\le x}
|\ell_j^{(d)}(t)|\right)^2
+ \sup_{t>x}|a(\sqrt{t})|.
\end{align*}
Let $\varepsilon>0$.
Using the assumption that $a(r)\to0$
as $r\to+\infty$, we choose $x$ such that the second summand is less than $\varepsilon/2$.
After that, applying Lemma~\ref{lem:Laguerre_tend_to_zero} with this fixed $x$,
we make the first summand less than $\varepsilon/2$.

2. If $v=0$, $j,k\in\bNz$, then
we obtain $\lim_{d\to+\infty}\beta_{a,d,j,k}=0$
by applying the Schwarz inequality
and the result of the first part of this proof.

3. For general $v$ in $\bC$,
we rewrite $a$ in the form
$(a-v 1_{(0,+\infty)})+v 1_{(0,+\infty)}$.
Since
\[
\beta_{1_{(0,+\infty)},d,j,k}
=\int_0^{+\infty}\ell_j^{(d)}(t) \ell_k^{(d)}(t)\,\dif{}t
=\delta_{j,k},
\]
the limit relation \eqref{eq:limit_beta} follows
from the result of the second part of this proof.
\end{proof}

\begin{prop}\label{prop:Cstar_Toeplitz_radial_converging_true_poly_Fock}
The C*-algebra $\cT_{(n)}(\RBC)$ is isometrically isomorphic to $c(\bNz)$.
\end{prop}

\begin{proof}
Recall that $\Phi_{(n)}$
is an isometrical isomorphism
$\cR_{(n)}\to\ell^\infty(\bNz)$
defined by~\eqref{eq:Phi_true}.
By Proposition~\ref{prop:radial_Toeplitz_eigenvalues_true_poly_Fock},
$\Phi_{(n)}(\{T_b\colon\ b\in\RBC\})=\fL$, where
\[
\fL\eqdef \{\lambda_{a,n}\colon\ a\in L^\infty([0,+\infty)),\ 
\exists v\in\bC\ \lim_{r\to+\infty}a(r)=v\}.
\]
So, $\cT_{(n)}(\RBC)$ is isometrically isomorphic
to the C*-subalgebra of $\ell^\infty(\bNz)$
generated by the set $\fL$.
By Lemma~\ref{lem:limit_beta}, $\fL\subseteq c(\bNz)$.
Our objective is to show that the C*-subalgebra of $c(\bNz)$
generated by $\fL$ coincides with $c(\bNz)$.
The space $c(\bNz)$ may be viewed as the C*-algebra
of the continuous functions on the compact $\bNz\cup\{+\infty\}$.
The set $\fL$ is a vector subspace of $c(\bNz)$
which contains the constants and is closed
under the pointwise conjugation.
In order to apply the Stone--Weierstrass theorem,
we have to prove that the set $\fL$
separates the points of $\bNz\cup\{+\infty\}$.
For every $u$ in $(0,+\infty]$,
define $a_u$ to be the characteristic function $1_{(0,u)}$.
Then
\[
\lambda_{a_u,n}(p)
=\int_0^{u^2}\,
\bigl(\ell_{\min\{p,n-1\}}^{(|p-n+1|)}(t)\bigr)^2\,\dif{}t.
\]
Let $p,q\in\bNz$, $p\ne q$.
If $\lambda_{a_u,n}(p)=\lambda_{a_u,n}(q)$
for all $u>0$, then for all $t>0$
\[
\bigl(\ell_{\min\{p,n-1\}}^{(|p-n+1|)}(t)\bigr)^2
=\bigl(\ell_{\min\{q,n-1\}}^{(|q-n+1|)}(t)\bigr)^2,
\]
which is not true.
So, the set $\fL$ separates $p$ and $q$.

Now let $p\in\bNz$ and $q=+\infty$.
Put $u=1$.
Then $\lambda_{a_1,n}(p)>0$, but
$\lambda_{a_1,n}(+\infty)=\lim_{r\to+\infty}a_1(r)=0$.
So, the set $\fL$ separates $p$ and $+\infty$.
\end{proof}

Recall that $\Phi_n\colon\cR_n\to\fM_n$ is defined by~\eqref{eq:Phi}.

\begin{prop}
Let $a\in L^\infty([0,+\infty))$.
Then $T_{n,\widetilde{a}}\in\cR_n$,
and the $d$-th component
of the sequence $\Psi(T_{n,\widetilde{a}})$
is the matrix
\[
\Psi(T_{n,\widetilde{a}})_d
=\bigl[ \beta_{a,d,j,k} \bigr]_{j,k=\max\{0,-d\}}^{n-1}.
\]
\end{prop}

\begin{proof}
Apply Corollary~\ref{cor:Toeplitz_with_invariant_symbol_is_invariant_under_group_representation}
and Proposition~\ref{prop:radial_multiplication}.
\end{proof}

Let $\fC_n$ be the C*-subalgebra of $\fM_n$
that consists of all matrix sequences that have scalar limits:
\[
\fC_n \eqdef \{A\in\fM_n\colon\quad
\exists v\in\bC\quad \lim_{d\to+\infty}A_d=v\,I_n\}.
\]

\begin{prop}
$\Phi_n(\cT_n(\RBC))\subseteq\fC_n$.
\end{prop}

\begin{proof}
Follows from Lemma~\ref{lem:limit_beta}.
\end{proof}

We finish this section with a couple of conjectures.

\begin{conj}\label{conj:Cstar_radial_in_true_polyanalytic_Fock}
The C*-algebra $\cT_{(n)}(\RB)$
is isometrically isomorphic to the C*-algebra of
bounded square-root-oscillating sequences.
\end{conj}

The concept of square-root-oscillating sequences
and a proof of Conjecture~\ref{conj:Cstar_radial_in_true_polyanalytic_Fock} for $n=1$ can be found in \cite{EsmeralMaximenko2016}.

\begin{conj}\label{conj:Cstar_radial_converging_in_polyanalytic_Fock}
$\Phi_n(\cT_n(\RBC))=\fC_n$.
\end{conj}

Various results, similar to Conjecture~\ref{conj:Cstar_radial_converging_in_polyanalytic_Fock},
but for Toeplitz operators in other spaces of functions
or with generating symbols invariant under other group actions, were proved by Loaiza, Lozano, Ram\'irez Ortega,
S\'anchez Nungaray,
Gonz\'alez-Flores, L\'opez-Mart\'inez,
and Arroyo-Neri
\cite{LoaizaLozano2014,RamirezSanchez2015,
LoaizaRamirez2017,SanchezGonzalezLopezArroyo2018}.

\section*{Acknowledgements}

The authors are grateful to the CONACYT (Mexico) scholarships
and to IPN-SIP projects (Instituto Polit\'{e}cnico Nacional, Mexico) for the financial support.
This research is inspired by many works of Nikolai Vasilevski.
We also thank Jorge Iv\'{a}n Correo Rosas
for discussions of the proof of Proposition~\ref{prop:basis_in_Fn}.

\noindent
Egor A. Maximenko\newline
Instituto Polit\'{e}cnico Nacional\newline
Escuela Superior de F\'{i}sica y Matem\'{a}ticas\newline
Ciudad de M\'{e}xico\newline
Mexico\newline
e-mail: emaximenko@ipn.mx\newline
https://orcid.org/0000-0002-1497-4338

\medskip\noindent
Ana Mar\'{i}a Teller\'{i}a-Romero\newline
Instituto Polit\'{e}cnico Nacional\newline
Escuela Superior de F\'{i}sica y Matem\'{a}ticas\newline
Ciudad de M\'{e}xico\newline
Mexico\newline
e-mail: anamariatelleriaromero@gmail.com\newline
https://orcid.org/0000-0002-9821-9398

\noindent
Apartado Postal 07730\newline
Apartado Postal 07730

\end{document}